\documentclass[11pt,twoside, leqno]{article}

\usepackage{mathrsfs}
\usepackage{amssymb}
\usepackage{amsmath}
\usepackage{amsthm}
\usepackage{amsfonts}
\usepackage{color}

\allowdisplaybreaks

\def\rr{{\mathbb R}}
\def\rn{{{\rr}^n}}
\def\zz{{\mathbb Z}}

\def\nn{{\mathbb N}}
\def\fz{\infty}

\def\aa{{\mathbb A }}

\def\cs{{\mathcal S}}

\def\cb{{\mathcal B}}

\def\az{\alpha}

\renewcommand\tilde{\widetilde}

\def\supp{{\rm{\,supp\,}}}
\def\esup{\mathop\mathrm{\,ess\,sup\,}}

\def\ls{\lesssim}
\def\lz{\lambda}
\def\Lz{\Lambda}

\def\bz{\beta}
\def\vz{\varphi}
\def\gz{{\gamma}}

\def\boz{\Omega}
\def\Oz{{\Omega}}

\def\hs{\hspace{0.3cm}}

\def\r{\right}
\def\lf{\left}

\def\bint{{\ifinner\rlap{\bf\kern.30em--}
\int\else\rlap{\bf\kern.35em--}\int\fi}\ignorespaces}

\def\sbint{{\ifinner\rlap{\bf\kern.32em--}
\hspace{0.078cm}\int\else\rlap{\bf\kern.45em--}\int\fi}\ignorespaces}

\def\dsup{\displaystyle\sup}

\newtheorem{theorem}{Theorem}[section]
\newtheorem{lemma}[theorem]{Lemma}
\newtheorem{corollary}[theorem]{Corollary}

\theoremstyle{definition}
\newtheorem{remark}[theorem]{Remark}
\newtheorem{definition}[theorem]{Definition}
\numberwithin{equation}{section}

\numberwithin{equation}{section}

\numberwithin{equation}{section}

\begin{document}

\arraycolsep=1pt

\title{\Large\bf Estimates for Parametric Marcinkiewicz Integrals on Musielak-Orlicz Hardy Spaces
\footnotetext{\hspace{-0.35cm}
{\it 2010 Mathematics Subject Classification}.
{Primary 42B20; Secondary 42B30, 46E30.}
\endgraf{\it Key words and phrases.} Marcinkiewicz integral, Muckenhoupt weight,
Musielak-Orlicz function, Hardy space.
\endgraf This work is partially supported by the National
Natural Science Foundation of China (Grant Nos. 11461065 \& 11661075) and A Cultivate Project for Young Doctor from Xinjiang Uyghur Autonomous Region (No. qn2015bs003).
 \endgraf $^\ast$\,Corresponding author
}}
\author{Liu Xiong, Li Baode, Qiu Xiaoli and Li Bo\,$^\ast$  \\
}
\date{  }
\maketitle

\vspace{-0.8cm}

\begin{minipage}{13cm}\small
{
\noindent
{\bf Abstract:}
Let $\varphi:\mathbb{R}^n\times[0,\,\infty) \rightarrow [0,\,\infty)$
satisfy that $\varphi(x,\,\cdot)$, for any given $x\in\mathbb{R}^n$,
is an Orlicz function and $\varphi(\cdot\,,t)$ is a Muckenhoupt $A_\infty$ weight
uniformly in $t\in(0,\,\infty)$. The Musielak-Orlicz Hardy space $H^\varphi(\mathbb{R}^n)$
generalizes both of the weighted Hardy space and the Orlicz Hardy space and hence has a wide generality.
In this paper, the authors first prove the completeness of both of the Musielak-Orlicz space
$L^\varphi(\mathbb{R}^n)$ and the weak Musielak-Orlicz space $WL^\varphi(\mathbb{R}^n)$. Then the authors obtain two boundedness criterions of
operators on Musielak-Orlicz spaces. As applications, the authors establish the boundedness of parametric Marcinkiewicz integral $\mu^\rho_\Omega$
from $H^\varphi(\mathbb{R}^n)$ to $L^\varphi(\mathbb{R}^n)$ (resp. $WL^\varphi(\mathbb{R}^n)$)
under weaker smoothness condition (resp. some Lipschitz condition) assumed on $\Omega$.
These results are also new even when $\varphi(x,\,t):=\phi(t)$ for all
$(x,\,t)\in\mathbb{R}^n\times[0,\,\infty)$, where $\phi$ is an Orlicz function.
}
\end{minipage}



\section{Introduction\label{s1}}
Suppose that $S^{n-1}$ is the unit sphere in the $n$-dimensional Euclidean space $\rn \ (n\ge2)$.
Let $\Omega$ be a {homogeneous function of degree zero} on $\rn$ which is locally integrable and satisfies the cancellation condition
\begin{align}\label{e1.1}
\int_{S^{n-1}}\Omega(x')\,d\sigma(x')=0,
\end{align}
where $d\sigma$ is the Lebesgue measure and $x':=x/{|x|}$ for any $x\neq{\mathbf{0}}$. For a function $f$ on $\rn$,
the parametric {Marcinkiewicz integral} $\mu^\rho_\Omega$ is defined by setting, for any $x\in\rn$ and $\rho\in(0,\,\infty)$,
$$\mu^\rho_\Omega(f)(x):=\lf(
\int_{0}^{\fz}\lf|F^\rho_{\Omega,\,t}(f)(x)\r|^2\,\frac{dt}{t^{2\rho+1}}\r)^{1/2},$$
where
$$F^\rho_{\Omega,\,t}(f)(x):=\int_{|x-y|\leq t}\frac{\Omega(x-y)}{|x-y|^{n-\rho}}f(y)\,{dy}.$$
When $\rho:=1$, we shall denote $\mu^1_\Omega$ simply by $\mu_\Omega$, which is reduced to the classic Marcinkiew-
icz integral. In 1938, Marcinkiewicz \cite{m38} first defined the operator $\mu_\Omega$ for $n=1$ and $\Omega(t):={\rm{sign}}\,t$.
The Marcinkiewicz integral of higher dimensions was studied by Stein \cite{s58} in 1958.
He showed that,
if $\Omega\in{\rm{Lip}}_\alpha(S^{n-1})$ with $\alpha\in(0,\,1]$, then $\mu_\Omega$ is
bounded on $L^p(\rn)$ with $p\in(1,\,2]$ and bounded from $L^1(\rn)$ to weak $L^1(\rn)$.
On the other hand, in 1960, H\"{o}rmander \cite{h60} proved that, if $\Omega\in{\rm{Lip}}_\alpha(S^{n-1})$
with $\alpha\in(0,\,1]$, then $\mu^\rho_\Omega$ is bounded on $L^p(\rn)$ provided that
$p\in(1,\,\fz)$ and $\rho\in (0,\,\infty)$.
Notice that all the results mentioned above hold true depending on some smoothness condition of $\Omega$.
However, in 2009, Jiang et al. \cite{sj09} obtained  the following celebrated result that
$\mu^\rho_\Omega$ is bounded on $L^p_\omega(\rn)$ without any smoothness condition of $\Omega$,
where $\omega\in A_p$ and $A_p$ denotes the Muckenhoupt weight class.
\begin{flushleft}
{\bf{Theorem A.}}
{\it{
Let $\rho\in(0,\,\infty)$, $q\in(1,\,\fz)$, $q':=q/(q-1)$ and $\Omega\in L^q(S^{n-1})$ satisfying \eqref{e1.1}.
If $\omega^{q'}\in A_p$ with $p\in(1,\,\fz)$, then there exists a positive constant $C$ independent of $f$ such that
$$\lf\|\mu^\rho_\Omega(f)\r\|_{L^p_\omega(\rn)}\le C\|f\|_{L^p_\omega(\rn)}.$$
}}
\end{flushleft}

It is now well known that Hardy space $H^p(\rn)$ is a good substitute of the Lebesgue space
$L^p(\rn)$ with $p\in(0,\,1]$ in the study for the boundedness of operators and hence,
in 2007, Lin et al. \cite{ll07}  proved that the $\mu_\Omega$ is bounded
from weighted Hardy space to weighted Lebesgue space under weaker smoothness condition assumed on $\Omega$,
which is called $L^q$-Dini type condition of order $\alpha$ with $q\in[1,\,\fz]$ and $\az\in(0,\,1]$ (see Section \ref{s4} below for its definition).
In 2016, Wang \cite{w16} discussed the boundedness of $\mu^\rho_\Omega$
from weighted Hardy space to weighted Lebesgue space or to weighted weak Lebesgue space if
$\Omega\in{\rm{Lip}}_\alpha(S^{n-1})$ with $\az\in(0,\,1]$.
More conclusions of Marcinkiewicz integral are referred to \cite{ak14,gahk16,lll17}.

On the other hand, recently, Ky \cite{k14} studied a new Hardy space called Musielak-Orlicz Hardy space $H^\vz(\rn)$,
which generalizes both of the weighted Hardy space (cf. \cite{st89}) and
the Orlicz Hardy space (cf. \cite{j80,jy10}), and hence has a wide generality.
Apart from interesting theoretical considerations, the motivation to study
$H^\vz(\rn)$ comes from applications to elasticity,
fluid dynamics, image processing, nonlinear PDEs and the calculus of
variation (cf. \cite{d05,d09}).
More Musielak-Orlicz-type spaces are referred to
\cite{lhy12,hyy13,ly13,ccyy16,lffy16,lsl16,fhly17,ylk17}.

In light of Lin \cite{ll07}, Wang \cite{w16} and Ky \cite{k14}, it is a natural and interesting problem to ask whether parametric Marcinkiewicz integral
$\mu^\rho_\Omega$ is bounded from $H^\vz(\rn)$ to $L^\vz(\rn)$ (resp. $WL^\varphi(\rn)$)
under weaker smoothness condition (resp. some Lipschitz condition) assumed on $\Omega$.
In this paper we shall answer this problem affirmatively.

Precisely, this paper is organized as follows.

In Section \ref{s2}, we recall some notions concerning Muckenhoupt weights, growth functions,
Musielak-Orlicz space $L^\vz(\rn)$ and weak Musielak-Orlicz space $WL^\vz(\rn)$. Then we establish
the completeness of $L^\vz(\rn)$ and $WL^\vz(\rn)$ (see Theorems \ref{lwbx} and \ref{wlwbx} below).

Section \ref{s3} is devoted to establishing two boundedness criterions of operators
from $H^\vz(\rn)$ to $L^\vz(\rn)$ or from $H^\vz(\rn)$ to $WL^\vz(\rn)$ (see Theorems \ref{yt} and \ref{yt2} below).
In the process of the proofs of Theorem \ref{yt} and Theorem
\ref{yt2},
the Aoki-Rolewicz theorem (see Lemma \ref{ardl} below) and the weak type superposition principle
(see Lemma \ref{dj} below) play indispensable roles, respectively.

In Section \ref{s4}, we obtain
the boundedness of $\mu^\rho_\Omega$
from $H^\vz(\rn)$ to $L^\vz(\rn)$ (resp. $WL^\varphi(\rn)$) under weaker smoothness condition (resp. some Lipschitz condition) assumed on $\Omega$
(see Theorem \ref{dingli.1},
Theorem \ref{dingli.2}, Corollary \ref{tuilun.1}, Theorem \ref{dingli.3}
and Theorem \ref{dingli.4} below).
In the process of the proof of Theorem \ref{dingli.1}, it is worth pointing out that, since the space variant $x$ and
the time variant $t$ appeared in $\vz(x,\,t)$ are inseparable,
we can not directly use the method of Lin \cite{ll07}.
This difficulty is overcame via establishing a more subtle pointwise estimate for $\mu^\rho_\Omega(b)$
 (see Lemma \ref{lemma.1} below for more details),
where $b$ is a multiple of an atom.

Finally, we make some conventions on notation.
Let $\zz_+:=\{1,\, 2,\,\ldots\}$ and $\nn:=\{0\}\cup\zz_+$.
For any $\bz:=(\bz_1,\ldots,\bz_n)\in\nn^n$,
let $|\bz|:=\bz_1+\cdots+\bz_n$.
Throughout this paper the letter $C$ will denote a \emph{positive constant} that may vary
from line to line but will remain independent of the main variables.
The \emph{symbol} $P\ls Q$ stands for the inequality $P\le CQ$. If $P\ls Q\ls P$, we then write $P\sim Q$.
For any sets $E,\,F \subset \rn$, we use $E^\complement$ to denote the set $\rn\setminus E$,
$|E|$ its  {\it{$n$-dimensional Lebesgue measure}},
$\chi_E$ its \emph{characteristic function} and
$E+F$ the {\it algebraic sum} $\{x+y:\ x\in E,\,y\in F\}$.
For any $s\in\rr$, $\lfloor s\rfloor$ denotes the
unique integer such that $s-1<\lfloor s\rfloor\le s$.
If there are no special instructions, any space $\mathcal{X}(\rn)$ is denoted simply by $\mathcal{X}$. For instance, $L^2(\rn)$ is simply denoted by $L^2$.
For any index $q\in[1,\,\fz]$, $q'$ denotes the {\it{conjugate index}} of $q$, namely, $1/q+1/{q'}=1$.
For any set $E$ of $\rn$, $t\in[0,\,\infty)$ and measurable function $\vz$,
let $\vz(E,\,t):=\int_E\vz(x,\,t)\,dx$ and $\{|f|>t\}:=\{x\in\rn: \ |f(x)|>t\}$.
As usual we use $B_r$ to denote the ball $\{x\in\rn:\ |x|<r\}$ with $r\in(0,\,\fz)$.




\section{Completeness of $L^\vz$ and $WL^\vz$}\label{s2}
In this section, we first recall some notions concerning Muckenhoupt weights, growth functions,
Musielak-Orlicz space $L^\vz$  and weak Musielak-Orlicz space $WL^\vz$, and then establish
the completeness of $L^\vz$ and $WL^\vz$.

Recall that a nonnegative function $\vz$ on $\rn\times[0,\,\fz)$ is called a {\it Musielak-Orlicz function} if,
for any $x\in\rn$, $\vz(x,\,\cdot)$ is an Orlicz function on $[0,\,\fz)$ and, for any $t\in[0,\,\fz)$,  $\vz(\cdot\,,t)$ is measurable on $\rn$.
Here a function $\phi: [0,\,\fz) \to [0,\,\fz)$ is called an {\it Orlicz function},
if it is nondecreasing, $\phi(0) = 0$, $\phi(t) > 0$
for any $t\in(0,\,\fz)$, and $\lim_{t\to\fz} \phi(t) = \fz$.

Given a Musielak-Orlicz function $\vz$ on $\rn\times[0,\,\fz)$,
$\vz$ is said to be of {\it{uniformly lower}} (resp. {\it{upper}}) {\it{type}} $p$ with $p\in(-\fz,\,\fz)$,
if there exists a positive constant $C:=C_{\vz}$ such that, for any $x\in\rn$, $t\in[0,\,\fz)$ and $s\in(0,\,1]$
(resp. $s\in[1,\,\fz)$),
\begin{eqnarray*}
\vz(x,\,st)\le C s^p\vz(x,\,t).
\end{eqnarray*}
The {\it critical uniformly lower type index}
and the {\it critical uniformly upper type index}
of $\vz$
are, respectively,
defined by
\begin{align}\label{e2.1}
i(\vz):=\sup\{ p\in(-\fz,\,\fz): \vz \mathrm{ \ is \ of \ uniformly\  lower\  type \ {\it p}} \},
\end{align}
and
\begin{align}\label{e2.1.1}
I(\vz):=\inf\{p\in(-\fz,\,\fz):\vz \mathrm{ \ is \ of \ uniformly\  upper\  type \ {\it p}} \}.
\end{align}
Observe that $i(\vz)$ or $I(\vz)$ may not be attainable,
namely, $\vz$ may not be of uniformly lower type $i(\vz)$ or
of uniformly upper type $I(\vz)$ (see \cite[p.\,415]{lhy12} for more details).

\begin{definition}\label{d2.2}
\begin{enumerate}
\item[\rm{(i)}]Let $q\in[1,\,\fz)$. A locally integrable function $\vz(\cdot\,,t): \rn \to [0,\,\fz)$ is said to satisfy
the {\it uniform Muckenhoupt condition} $\aa_q$,
denoted by $\vz\in \aa_q$, if there exists a positive constant $C$ such that,
for any ball $B\subset\rn$ and $t\in(0,\,\fz)$, when $q=1$,
$$\frac{1}{|B|}\int_{B} \vz(x,\,t)\,dx\lf\{\esup_{x\in B} [\vz(x,\,t)]^{-1}\r\}\le C$$
and, when $q\in(1,\fz)$,
$$\frac{1}{|B|}\int_{B}\vz(x,\,t)\,dx
\lf\{\frac{1}{|B|}\int_{B}[\vz(x,\,t)]^{-\frac{1}{q-1}}\,dx\r\}^{q-1}
\le C.$$

\item[\rm{(ii)}]
Let $q\in(1,\,\fz]$. A locally integrable function $\vz(\cdot\,,t): \rn \to [0,\,\fz)$ is said to satisfy
the {\it{uniformly reverse H\"{o}lder condition}} $\mathbb{RH}_q$,
denoted by $\vz\in \mathbb{RH}_q$, if there exists a positive constant $C$ such that,
for any ball $B\subset\rn$ and $t\in(0,\,\fz)$, when $q\in(1,\fz)$,
$$\lf\{\frac{1}{|B|}\int_B\vz(x,\,t)\,dx\r\}^{-1}\lf\{\frac{1}{|B|}\int_B[\vz(x,\,t)]^q\,dx\r\}^{1/q}\leq C$$
and, when $q=\fz$,
$$\lf\{\frac{1}{|B|}\int_B\vz(x,\,t)\,dx\r\}^{-1}\esup_{x\in B} \vz(x,\,t)\,\leq C.$$
\end{enumerate}
\end{definition}

Define $\aa_\fz:=\bigcup_{q\in[1,\,\fz)} \aa_q$
and, for any $\vz\in\aa_\fz$,
\begin{eqnarray}\label{e2.4}
q(\vz):=\inf\{q\in[1,\,\fz):\ \vz\in\aa_q\}.
\end{eqnarray}
Observe that, if $q(\vz)\in(1,\,\fz)$, then $\vz\notin\aa_{q(\vz)}$, and there exists $\vz\notin\aa_1$ such that $q(\vz)=1$ (cf. \cite{jn87}).

\begin{definition}\label{d2.3}{\rm\cite[Definition 2.1]{k14}}
A function $\vz: \rn\times[0,\,\fz) \to [0,\,\fz)$ is  called a {\it{growth function}}
if the following conditions are satisfied:
\begin{enumerate}
\item[\rm{(i)}] $\vz$ is a {Musielak-Orlicz function};
%
%
\item[\rm{(ii)}] $\vz\in\aa_\fz$;

\item[\rm{(iii)}] $\vz$ is of uniformly lower type $p$ for some $p\in(0,\,1]$ and of
 uniformly upper type $1$.
\end{enumerate}
\end{definition}

Suppose that $\vz$ is a Musielak-Orlicz function. Recall that the
\emph{Musielak-Orlicz space} $L^{\vz}$ is defined to be the set of all measurable functions $f$ such that,
for some $\lambda\in(0,\,\fz)$,
$$\int_\rn \vz\lf(x,\, \frac{|f(x)|}{\lz}\r)\, dx<\fz$$ equipped with the Luxembourg-Nakano (quasi-)norm
$$\|f\|_{L^\vz}:=\inf\lf\{ \lz\in(0,\,\fz):\ \int_\rn \vz\lf(x,\, \frac{|f(x)|}{\lz}\r)\, dx\le 1\r\}.$$

Similarly, the \emph{weak Musielak-Orlicz space} $WL^{\vz}$ is defined to be the set of all measurable functions $f$ such that, for some $\lambda\in(0,\,\fz)$,
$$\sup_{t\in(0,\,\fz)} \vz\lf(\{|f|>t\},\, \frac{t}{\lz}\r)<\fz$$ equipped with the quasi-norm
$$ \|f\|_{WL^\vz}:=\inf\lf\{ \lz\in(0,\,\fz):\ \sup_{t\in(0,\,\fz)}\vz\lf(\{|f|>t\},\,\frac{t}{\lz}\r)\le 1\r\}. $$

\begin{remark}\label{r2.11}
Let $\omega$ be a classic Muckenhoupt weight and $\phi$ an Orlicz function.
\begin{enumerate}
\item[(i)] If $\vz(x,\,t):=\omega(x)t^{p}$ for all $(x,\,t)\in \rn\times[0,\,\infty)$ with $p\in(0,\,\fz)$,
then $L^\varphi$ (resp. $WL^\vz$) is reduced to weighted Lebesgue space $L_\omega^p$
(resp. weighted weak Lebesgue space $WL^p_\omega$), and particularly, when $\omega\equiv 1$, the corresponding unweighted spaces are also obtained.
\item[(ii)] If $\vz(x,\,t):=\omega(x)\phi(t)$ for all $(x,\,t)\in \rn\times[0,\,\infty)$, then $L^\varphi$
(resp. $WL^\vz$) is reduced to weighted Orlicz space $L_\omega^\phi$
(resp. weighted weak Orlicz space $WL^\phi_\omega$), and particularly, when $\omega\equiv 1$, the corresponding unweighted spaces are also obtained.
\end{enumerate}
\end{remark}
Throughout the paper, we always assume that $\varphi$ is a growth function.

In order to obtain the completeness of ${L}^\varphi$, we need the following several lemmas,
which are some properties of growth functions.
\begin{lemma}\label{qcytj}{\rm{\cite[Lemma 4.2]{k14}}}
Let $\varphi$ be a growth function as in Definition \ref{d2.3}. Then the following hold true:
\begin{enumerate}
\item[\rm{(i)}]
for any $f\in {L}^\varphi$ satisfying $f\not\equiv 0$,
$$\int_{\rn}\vz\lf(x,\,\frac{|f(x)|}{\|f\|_{L^\vz}}\r)\,dx=1;$$
\item[\rm{(ii)}]
$\mathop{\lim}\limits_{k\rightarrow\fz}\|f_k\|_{{L}^\varphi}=0$ if and only if
$\mathop{\lim}\limits_{k\rightarrow\fz}\int_{\rn}\vz\lf(x,\,|f_k(x)|\r)\,dx=0.$
\end{enumerate}
\end{lemma}

The following lemma comes from \rm{\cite[Lemma 4.1]{k14}},
and also can be found in \cite{ylk17}.

\begin{lemma}\label{ckj}
Let $\vz$ be a growth function as in Definition \ref{d2.3}.
Then there exists a positive constant $C$ such that,
for any $(x,\,t_j)\in\rn \times [0,\,\fz)$ with $j\in\zz_+$,
$$\vz\lf(x,\,\sum_{j=1}^\fz t_j\r)\le C\sum_{j=1}^\fz\vz\lf(x,\,t_j\r).$$
\end{lemma}

\begin{theorem}\label{lwbx}
The space $L^\varphi$ is complete.
\end{theorem}

\begin{proof}
In order to prove the completeness of $L^\vz$, it suffices to prove that,
for any sequence $\{f_j\}_{j\in\zz_+}\subset L^\varphi$ satisfying $\|f_j\|_{L^\varphi}\leq 2^{-j}$,
the series $\{\sum_{j=1}^k f_j\}_{k\in\zz_+}$ converges in $L^\vz$.
By the uniformly lower type $p$ property of $\vz$ and Lemma \ref{qcytj}(i), we see that, for any $j\in\zz_+$,
\begin{align}\label{ee.1}
\int_{\rn}\vz(x,\,|f_j(x)|)\,dx\le\int_{\rn}\vz\lf(x,\,2^{-j}\frac{|f_j(x)|}{\|f_j\|_{L^\vz}}\r)\,dx\ls2^{-jp}.
\end{align}
Noticing that the series $\{\sum_{j=1}^k f_j\}_{k\in\zz_+}$ is a Cauchy sequence in $L^\varphi$, we have
$$\lim_{k,\,m\rightarrow\fz}\lf\|{\sum_{j=1}^k f_j}-{\sum_{j=1}^m f_j}\r\|_{L^\varphi}=0,$$
which, together with Lemma \ref{qcytj}(ii), implies that
\begin{align}\label{eea.1}
\mathop{\lim}\limits_{k,\,m\rightarrow\fz}\int_{\rn}
\vz\lf(x,\,\lf|{\sum_{j=1}^k f_j(x)}-{\sum_{j=1}^m f_j(x)}\r|\r)\,dx=0.
\end{align}
By the uniformly lower type $p$ and the uniformly upper type $1$ properties of $\varphi$, and \eqref{eea.1}, we know that, for any $\sigma\in(0,\,\infty)$,
\begin{align*}
&\mathop{\lim}\limits_{k,\,m\rightarrow\fz}\vz\lf(\lf\{\lf|{\sum_{j=1}^k f_j}-{\sum_{j=1}^m f_j}\r|>\sigma\r\},\,1\r)\\
&\hs=\mathop{\lim}\limits_{k,\,m\rightarrow\fz}\vz\lf(\lf\{\lf|{\sum_{j=1}^k f_j}-{\sum_{j=1}^m f_j}\r|>\sigma\r\},\,\frac{1}{\sigma}\,\sigma\r)\\
&\hs\lesssim \max\lf\{\sigma^{-1},\sigma^{-p}\r\} \mathop{\lim}\limits_{k,\,m\rightarrow\fz}\vz\lf(\lf\{\lf|{\sum_{j=1}^k f_j}-{\sum_{j=1}^m f_j}\r|>\sigma\r\},\,\sigma\r)\\
&\hs\lesssim \max\lf\{\sigma^{-1},\sigma^{-p}\r\}\mathop{\lim}\limits_{k,\,m\rightarrow\fz}\int_{\lf\{\lf|{\sum_{j=1}^k f_j}-{\sum_{j=1}^m f_j}\r|>\sigma\r\}}
\vz\lf(x,\,\lf|{\sum_{j=1}^k f_j(x)}-{\sum_{j=1}^m f_j(x)}\r|\r)\,dx\\
&\hs\lesssim \max\lf\{\sigma^{-1},\sigma^{-p}\r\} \mathop{\lim}\limits_{k,\,m\rightarrow\fz}\int_{\rn}
\vz\lf(x,\,\lf|{\sum_{j=1}^k f_j(x)}-{\sum_{j=1}^m f_j(x)}\r|\r)\,dx\thicksim0.
\end{align*}
Hence, there exists some $f$ such that $\sum_{j=1}^k f_j$ converges to $f$ as $k\rightarrow\infty$ in measure.
From this and using Riesz's theorem, we deduce that there exists a subsequence
$\sum_{j=1}^{k_i} f_{j}\rightarrow f$ as $i\rightarrow \infty$ almost everywhere.
By this, Lemma \ref{ckj} and \eqref{ee.1}, we obtain that
\begin{align*}
\int_{\rn}\vz\lf(x,\,\lf|f(x)-{\sum_{j=1}^{k_i} f_j(x)}\r|\r)\,dx
\lesssim\sum_{j\geq\,{k_i}+1}\int_{\rn}\vz\lf(x,\,\lf|f_j(x)\r|\r)\,dx
\lesssim\sum_{j\geq\,{k_i}+1}2^{-jp}\rightarrow 0
\end{align*}
as $i\rightarrow\fz$. From Lemma \ref{qcytj}(ii) again, it follows that
$\mathop{\lim}\limits_{i\rightarrow\fz}\|f-{\sum_{j=1}^{k_i} f_j}\|_{L^\varphi}=0.$
On the other hand, noticing that $\{\sum_{j=1}^k f_j\}_{k\in\zz_+}$ is a Cauchy sequence in $L^\vz$,
then it is easy to see that $\lim_{k\rightarrow\fz}\|\sum_{j=1}^k f_j-f\|_{L^\vz}=0$ and $f\in L^\vz$.
This finishes the proof of Theorem \ref{lwbx}.
\end{proof}

In order to obtain the completeness of ${WL}^\varphi$, we need the following several lemmas,
which are some properties of growth functions.

\begin{lemma}\label{jxsqj}
Let $\varphi$ be a growth function as in Definition \ref{d2.3}. Then the following hold true:
\begin{enumerate}
\item[\rm{(i)}]{\rm{\cite[Lemma 3.3(ii)]{lyj16}}}
for any $f\in {WL}^\varphi$ satisfying $f\not\equiv 0$,
$$\sup_{t\in(0,\,\infty)}\varphi\lf(\{|f|>t\},\,\frac{t}{\|f\|_{{WL}^\varphi}}\r)=1;$$
\item[\rm{(ii)}]
$\mathop{\lim}\limits_{k\rightarrow\fz}\|f_k\|_{{WL}^\varphi}=0$ if and only if
$\mathop{\lim}\limits_{k\rightarrow\fz}\mathop{\sup}\limits_{t\in(0,\,\infty)}
\varphi\lf(\{|f_k|>t\},\,t\r)=0.$
\end{enumerate}
\end{lemma}

\begin{proof}
We only prove $\rm{(ii)}$ of Lemma \ref{jxsqj} since $\rm{(i)}$ of Lemma \ref{jxsqj} was proved in {\rm{\cite[Lemma 3.3(ii)]{lyj16}}}.
By the uniformly lower type $p$ and the uniformly upper type $1$ properties of $\varphi$, we conclude that, for any
$x\in\mathbb{R}^n$, $s\in(0,\,\infty)$ and $t\in(0,\,\infty)$,
\begin{align}\label{max}
\varphi\lf(x,\,st\r)\lesssim \max\lf\{s,s^p\r\}\varphi\lf(x,\,t\r)
\end{align}
and
\begin{align}\label{min}
\varphi\lf(x,\,st\r)\gtrsim \min\lf\{s,s^p\r\}\varphi\lf(x,\,t\r).
\end{align}
Thus, from \eqref{max} and Lemma \ref{jxsqj}$\rm{(i)}$, we deduce that
\begin{align*}
&\mathop{\sup}\limits_{t\in(0,\,\infty)}\varphi\lf(\{|f_k|>t\},\,t\r) \\
&\hs=\mathop{\sup}\limits_{t\in(0,\,\infty)}\varphi\lf(\{|f_k|>t\},\,{\|f_k\|_{{WL}^\varphi}}\frac{t}{\|f_k\|_{{WL}^\varphi}}\r) \\
&\hs\lesssim \max\lf\{{\|f_k\|_{{WL}^\varphi}},{\|f_k\|^p_{{WL}^\varphi}}\r\}
\mathop{\sup}\limits_{t\in(0,\,\infty)}\varphi\lf(\{|f_k|>t\},\,\frac{t}{\|f_k\|_{{WL}^\varphi}}\r)\\
&\hs\sim \max\lf\{{\|f_k\|_{{WL}^\varphi}},{\|f_k\|^p_{{WL}^\varphi}}\r\}.
\end{align*}
On the other hand, by \eqref{min} and Lemma \ref{jxsqj}$\rm{(i)}$, we obtain that
\begin{align*}
&\mathop{\sup}\limits_{t\in(0,\,\infty)}\varphi\lf(\{|f_k|>t\},\,t\r) \\
&\hs=\mathop{\sup}\limits_{t\in(0,\,\infty)}\varphi\lf(\{|f_k|>t\},\,{\|f_k\|_{{WL}^\varphi}}\frac{t}{\|f_k\|_{{WL}^\varphi}}\r) \\
&\hs\gtrsim \min\lf\{{\|f_k\|_{{WL}^\varphi}},{\|f_k\|^p_{{WL}^\varphi}}\r\}
\mathop{\sup}\limits_{t\in(0,\,\infty)}\varphi\lf(\{|f_k|>t\},\,\frac{t}{\|f_k\|_{{WL}^\varphi}}\r)\\
&\hs\sim\min\lf\{{\|f_k\|_{{WL}^\varphi}},{\|f_k\|^p_{{WL}^\varphi}}\r\}.
\end{align*}
From the above two inequalities, it follows that
$$\min\lf\{{\|f_k\|_{{WL}^\varphi}},{\|f_k\|^p_{{WL}^\varphi}}\r\}\lesssim
\mathop{\sup}\limits_{t\in(0,\,\infty)}\varphi\lf(\{|f_k|>t\},\,t\r)\lesssim
\max\lf\{{\|f_k\|_{{WL}^\varphi}},{\|f_k\|^p_{{WL}^\varphi}}\r\},$$
which implies that $\rm{(ii)}$ of Lemma \ref{jxsqj} holds true. This finishes the proof of Lemma \ref{jxsqj}.
\end{proof}

\begin{lemma}\label{jxxqj}{\rm{\cite[p.\,10]{g09c}}}
Let $\varphi$ be a growth function as in Definition \ref{d2.3}. If
$\mathop{\liminf}\limits_{k\rightarrow\fz}|f_k|=|f|$ almost everywhere, then, for any $t\in(0,\,\infty)$,
$$\varphi\lf(\{|f|>t\},\,t\r)\leq \mathop{\liminf}\limits_{k\rightarrow\fz} \varphi\lf(\{|f_k|>t\},\,t\r).$$
\end{lemma}

\begin{theorem}\label{wlwbx}
The space ${WL}^\varphi$ is complete.
\end{theorem}

\begin{proof}
To prove that ${WL}^\varphi$ is complete, take $\{f_k\}_{k\in\zz_+}\subset{{WL}^\varphi}$ such that
$\mathop{\lim}\limits_{k,\,m\rightarrow\fz}\|f_k-f_m\|_{{WL}^\varphi}=0$.
By Lemma \ref{jxsqj}(ii), we know that, for any chosen positive number $\varepsilon$, however small,
there exists a positive integer $K$ such that, whenever $k,\,m\in[K,\,\fz)\cap\zz_+$, then
\begin{align}\label{eea.2}
\sup_{t\in(0,\,\fz)}\vz(\{|f_k-f_m|>t\},\,t)<\varepsilon.
\end{align}
By the uniformly lower type $p$ and the uniformly upper type $1$ properties of $\varphi$, and \eqref{eea.2}, we know that, for any $\sigma\in(0,\,\infty)$,
\begin{align*}
\mathop{\lim}\limits_{k,\,m\rightarrow\fz}\vz\lf(\lf\{|f_k-f_m|>\sigma\r\},\,1\r)
&= \mathop{\lim}\limits_{k,\,m\rightarrow\fz}\vz\lf(\lf\{|f_k-f_m|>\sigma\r\},\,\frac{1}{\sigma}\,\sigma\r)\\
&\lesssim \max\lf\{\sigma^{-1},\sigma^{-p}\r\} \mathop{\lim}\limits_{k,\,m\rightarrow\fz}\vz\lf(\lf\{|f_k-f_m|>\sigma\r\},\,\sigma\r)\thicksim0.
\end{align*}
Hence, there exists some $f$ such that $f_k\rightarrow f$ as $k\rightarrow\infty$ in measure,
which, together with Riesz's theorem, implies that some subsequence
\begin{align}\label{ee.2}
f_{k_s}\rightarrow f \ {\rm{as}} \ s\rightarrow\infty \ {\rm{almost \ everywhere}}.
\end{align}
For the $K$ mentioned above, take $J\in\zz_+$ such that, for any $j\in[J,\,\fz)\cap\zz_+$, the positive integer $k_j\ge K$.
By \eqref{ee.2} and Lemma \ref{jxxqj}, we know that, there exists a positive integer $J$ such that, whenever $j\in[J,\,\fz)\cap\zz_+$, then
\begin{align*}
\sup_{t\in(0,\,\fz)}\vz(\{|f_{k_j}-f|>t\},\,t)
&=\sup_{t\in(0,\,\fz)}\vz\lf(\lf\{\lim_{s\rightarrow\fz}|f_{k_j}-f_{k_s}|>t\r\},\,t\r) \\
&\le\lim_{s\rightarrow\fz}\sup_{t\in(0,\,\fz)}\vz\lf(\lf\{|f_{k_j}-f_{k_s}|>t\r\},\,t\r)<\varepsilon,
\end{align*}
that is to say,
\begin{align*}
\lim_{j\rightarrow\fz}\sup_{t\in(0,\,\fz)}\vz(\{|f_{k_j}-f|>t\},\,t)=0.
\end{align*}
Applying Lemma \ref{jxsqj}(ii) again, we conclude that
$$\lim_{j\rightarrow\fz}\|f_{k_j}-f\|_{WL^\vz}=0.$$
On the other hand, noticing that $\{f_k\}_{k\in\zz_+}$ is a Cauchy sequence in $WL^\vz$,
then it is easy to see that $\lim_{k\rightarrow\fz}\|f_k-f\|_{WL^\vz}=0$ and $f\in WL^\vz$.
This finishes the proof of Theorem \ref{wlwbx}.
\end{proof}

\section{Two boundedness criterions of operators}\label{s3}
In this section, we first recall the notion concerning the Musielak-Orlicz Hardy space $H^\vz$
via the non-tangential grand maximal function, and then establish
two boundedness criterions of operators from $H^\vz$ to $L^\vz$ or from $H^\vz$ to $WL^\vz$.

In what follows, we denote by $\cs$ the {\it set of all Schwartz functions}
and by $\cs'$ its {\it dual space } (namely, the {\it set of all tempered
distributions}). For any $m\in\nn$, let $\cs_m$ be the {\it{set}} of all $\psi\in\cs$ such that $\|\psi\|_{\cs_m}\le1$, where
\begin{eqnarray*}
\|\psi\|_{\cs_m}:=\sup_{\az\in\nn^n,\,|\az|\le m+1}\sup_{x\in\rn}(1+|x|)^{(m+2)(n+1)}|\partial^\az\psi(x)|.
\end{eqnarray*}
Then, for any $m\in\nn$ and $f\in \cs'$, the {\it{non-tangential grand maximal function}}
$f^\ast_m$ of $f$ is defined by setting, for all $x\in\rn$,
\begin{eqnarray}\label{e2.m1}
f^\ast_m(x) := \sup_{\psi\in \cs_m}\,\sup_{|y-x|<t,\,t\in(0,\,\fz)}
|f\ast\psi_t(y)|,
\end{eqnarray}
where, for any $t\in(0,\,\fz)$, $\psi_t(\cdot):= t^{-n}\psi(\frac{\cdot}{t})$.
When
\begin{align}\label{e2.5}
m=m(\vz) :=\lf\lfloor n\lf(\frac{q(\vz)}{i(\vz)}-1\r)\r\rfloor,
\end{align}
we denote $f^\ast_m$ simply by $f^\ast$,
where $q(\vz)$ and $i(\vz)$ are as in \eqref{e2.4} and \eqref{e2.1}, respectively.

\begin{definition}\label{d2.5} \cite[Definition 2.2]{k14}
Let $\vz$ be a growth function as in Definition \ref{d2.3}.
The \emph{Musielak-Orlicz Hardy space} $H^\vz$ is defined as the set of all $f\in\cs'$
such that $f^\ast\in L^\vz$ endowed with the (quasi-)norm
$$\|f\|_{H^\vz}:=\|f^\ast\|_{L^\vz}.$$
\end{definition}

\begin{remark}\label{r2.111}
Let $\omega$ be a classic Muckenhoupt weight and $\phi$ an Orlicz function.
\begin{enumerate}
\item[(i)] If $\vz(x,\,t):=\omega(x)t^{p}$ for all $(x,\,t)\in \rn\times[0,\,\infty)$ with $p\in(0,\,1]$,
then $H^\varphi$ is reduced to weighted Hardy space $H_\omega^p$,
and particularly, when $\omega\equiv 1$, the corresponding unweighted space is also obtained.
\item[(ii)] If $\vz(x,\,t):=\omega(x)\phi(t)$ for all $(x,\,t)\in \rn\times[0,\,\infty)$, then $H^\varphi$ is reduced to weighted Orlicz Hardy space $H_\omega^\phi$, and particularly, when $\omega\equiv 1$, the corresponding unweighted space is also obtained.
\end{enumerate}
\end{remark}

\begin{definition}\label{d2.11}{\rm\cite[Definition 2.4]{k14}}
Let $\vz$ be a growth function as in Definition \ref{d2.3}.
\begin{enumerate}
\item[\rm{(i)}] A triplet $(\vz,\,q,\,s)$ is said to be {\it {admissible}},
if $q\in (q(\vz),\,\fz]$ and $s \in [m(\vz),\,\fz)\cap\nn$,
where $q(\vz)$ and $m(\vz)$ are as in \eqref{e2.4} and \eqref{e2.5}, respectively.

\item[\rm{(ii)}] For an admissible triplet $(\vz,\,q,\,s)$, a measurable function $a$ is called a
{\it $(\vz,\,q,\,s)$-atom} if there exists some ball $B\subset\rn$ such that the following conditions are satisfied:

\quad(a) $a$ is supported in $B$;

\quad(b) $\|a\|_{L^q_\vz(B)}\leq\|\chi_B\|^{-1}_{L^\vz}$, where
\begin{eqnarray*}
\|a\|_{L_\vz^q(B)}:=
\lf\{\begin{array}{ll}
\dsup_{t\in(0,\,\fz)}\lf[\frac{1}{\vz(B,\,t)}\int_B|a(x)|^q \vz(x,\,t)\,dx\r]^{1/q}
                          ,\,\,\,&q\in[1,\,\fz),\\
\,\\
\|a\|_{L^\fz(B)},&q=\fz;
\end{array}\r.
\end{eqnarray*}

\quad(c) $\int_\rn a(x)x^\az dx=0$ for any $\az\in\nn^n$ with $|\az|\leq s$.

\item[\rm{(iii)}] For an admissible triplet $(\vz,\,q,\,s)$,
the {\emph{Musielak-Orlicz atomic Hardy space}} $H^{\vz,\,q,\,s}_{\rm{at}}$ is defined as the set of all
$f \in \cs'$ which can be represented as a linear combination of $(\vz,\, q,\, s)$-atoms, that is,
$f =\sum_j b_j$ in $\cs'$, where $b_j$ for each $j$ is a multiple of some $(\vz,\, q,\, s)$-atom
supported in some ball ${x_j+B_{r_j}}$, with the property
$$\sum_{j}\vz\lf({x_j+B_{r_j}},\, \|b_j\|_{L^q_\vz({x_j+B_{r_j}})}\r)<\fz. $$
For any given sequence of multiples of $(\vz,\,q,\,s)$-atoms, $\{b_j\}_j$, let
$$ \Lz_q(\{b_j\}_j):=\inf\lf\{\lz\in(0,\,\fz):\ \sum_j \vz\lf({x_j+B_{r_j}},\,\frac{\|b_j\|_{L^q_\vz({x_j+B_{r_j}})}}{\lz}\r)\le 1 \r\}  $$
and then the (quasi-)norm of $f\in\cs'$ is defined by
$$\|f\|_{H^{\vz,\,q,\,s}_{\rm{at}}}:=\inf\lf\{\Lz_q\lf(\{b_j\}_j\r)\r\},  $$
where the infimum is taken over all admissible decompositions of $f$ as above.
\end{enumerate}
\end{definition}

We refer the readers to \cite{k14} and \cite{ylk17} for more details on the real-variable theory
of Musielak-Orlicz Hardy spaces.

\begin{definition}\label{d2.zh}
Let $X$ and $Y$ be two linear spaces. An operator $T$: $D\subset X\rightarrow Y$
is called a positive sublinear operator if, for any $x\in\mathbb{R}^n$, the following conditions are satisfied:
\begin{enumerate}
\item[\rm{(i)}] $T(f)(x)\geq0$;
\end{enumerate}
\begin{enumerate}
\item[\rm{(ii)}] $T(\alpha f)(x)\leq|\alpha|T(f)(x)$, where $\alpha\in\mathbb{R}$;
\end{enumerate}
\begin{enumerate}
\item[\rm{(iii)}] $T(f+g)(x)\leq T(f)(x)+T(g)(x)$.
\end{enumerate}
\end{definition}

\begin{lemma}\label{szxz}
Let $X$ and $Y$ be two linear spaces and $T: D\subset X\rightarrow Y$
be a positive sublinear operator as in Definition \ref{d2.zh}. Then, for any $f,g\in D$,
$$\lf|T(f)-T(g)\r|\leq T(f-g).$$
\end{lemma}

\begin{proof}
By Definition \ref{d2.zh}(ii), we obtain that
$$T(-f)\leq|-1|\,T(f)=T(f)\leq|-1|\,T\lf(-f\r)=T\lf(-f\r),$$
therefore, $T(-f)=T(f)$. Moreover, by Definition \ref{d2.zh}(iii), we know that
\begin{align*}
T(f)-T(g)=T(f-g+g)-T(g)
\leq T(f-g)+T(g)-T(g)=T(f-g)
\end{align*}
and
\begin{align*}
T(g)-T(f)&=T(g-f+f)-T(f)
\leq T(g-f)+T(f)-T(f)=T(g-f).
\end{align*}
From the above two inequalities and $T(-f)=T(f)$, we deduce that $\lf|T(f)-T(g)\r|\leq T(f-g)$.
This finishes the proof of Lemma \ref{szxz}.
\end{proof}

The following two lemmas come from \rm{\cite[Lemma 4.3(i), Theorem 3.1]{k14}},
respectively, and also can be found in \cite{ylk17}.

\begin{lemma}\label{ky4.3}
Let $\vz$ be a growth function as in Definition \ref{d2.3}.
For a given positive constant $\tilde{C}$,
there exists a positive constant $C$ such that, for any $\lz\in(0,\,\fz)$,
$$\int_\rn\vz\lf(x,\,\frac{|f(x)|}{\lambda}\r)\,dx\le\tilde{C} \ implies \ that \ \|f\|_{L^\vz}\le C \lambda.$$
\end{lemma}

\begin{lemma}\label{yzfj}
Let $(\vz,\,q,\,s)$ be an admissible triplet as in Definition \ref{d2.11}. Then
$$H^\vz=H^{\vz,\,q,\,s}_{{\rm{at}}}$$
with equivalent (quasi-)norms.
\end{lemma}

\begin{lemma}\label{cm}{\rm{\cite[Remark 4.1.4]{ylk17}}}
Let $\vz$ be a growth function as in Definition \ref{d2.3}.
Then $H^\vz\cap L^2$ is dense in $H^\vz$.
\end{lemma}

Recall that a {\it quasi-Banach space} $\mathcal{B}$ is a linear space endowed with a quasi-norm
$\|\cdot\|_\mathcal{B}$ which is nonnegative, non-degenerate (i.e., $\|f\|_\mathcal{B}=0$
if and only if $f={\bf 0}$), homogeneous, and obeys the quasi-triangle inequality, i.e., there
exists a constant $K$ no less than 1 such that, for any $f, g\in\mathcal{B}$,
$\|f+g\|_\mathcal{B}\leq K\lf(\|f\|_\mathcal{B}+\|g\|_\mathcal{B}\r)$.

\begin{lemma}\label{ardl}{\rm{\cite[Aoki-Rolewicz theorem]{s84}}}
Let $\mathcal{B}$ be a quasi-Banach space and $K$ a constant associated with $\cb$ as above.
Then, for any $f,\,g\in\mathcal{B}$,
$$\|f+g\|^\gamma_\mathcal{B}\leq\|f\|^\gamma_\mathcal{B}+\|g\|^\gamma_\mathcal{B},$$
where $\gamma:=\lf[\log_{2}(2K)\r]^{-1}$.
\end{lemma}

\begin{lemma}\label{nfbds}
Let $\mathcal{B}$ be a quasi-Banach space equipped with the quasi-norm $\|\cdot\|_\mathcal{B}$.
For any $\{f_k\}_{k\in\zz_+}\subset\mathcal{B}$ and $f\in\mathcal{B}$,
if $\mathop{\lim}\limits_{k\rightarrow\fz} \lf\|f_k-f\r\|_\mathcal{B}=0$, then
$$\lim_{k\rightarrow\fz}\lf\|f_k\r\|_\mathcal{B}=\lf\|f\r\|_\mathcal{B}.$$
\end{lemma}

\begin{proof}
By Lemma \ref{ardl}, we obtain that, for any $k\in\zz_+$
$$\lf\|f_k\r\|^\gamma_\mathcal{B}-\lf\|f\r\|^\gamma_\mathcal{B}
=\lf\|f_k-f+f\r\|^\gamma_\mathcal{B}-\lf\|f\r\|^\gamma_\mathcal{B}
\leq\lf\|f_k-f\r\|^{\gamma}_\mathcal{B},$$
where $\gz$ is a harmless constant as in Lemma \ref{ardl}.
Similarly, we have
$\lf\|f\r\|^\gamma_\mathcal{B}-\lf\|f_k\r\|^\gamma_\mathcal{B}\leq\lf\|f-f_k\r\|^{\gamma}_\mathcal{B},$
which, together with the above inequality, implies that
$$\lf|\lf\|f_k\r\|^\gamma_\mathcal{B}-\lf\|f\r\|^\gamma_\mathcal{B}\r|\le\lf\|f_k-f\r\|^{\gamma}_\mathcal{B}\rightarrow0 \ {\rm{as}} \ k\rightarrow\fz.$$
This finishes the proof of Lemma \ref{nfbds}.
\end{proof}

The following theorem gives a boundedness criterion of operators from $H^\vz$ to $L^\vz$.
\begin{theorem}\label{yt}
Let $\vz$ be a growth function as in Definition \ref{d2.3}.
Suppose that a linear or a positive sublinear operator $T$ is bounded on $L^2$.
If there exists a positive constant $C$ such that,
for any $\lz\in(0,\,\fz)$ and multiple of a $(\vz,\,q,\,s)$-atom $b$ associated with some ball $B\subset\rn$,
\begin{align}\label{a1}
 \int_{\rn}\vz\lf(x,\,\frac{|T(b)(x)|}{\lz}\r)\,dx \le C\vz\lf(B,\,\frac{\|b\|_{L^q_\vz(B)}}{\lz}\r),
\end{align}
then $T$ extends uniquely to a bounded operator from ${H^\vz}$ to $L^\vz$.
\end{theorem}
\begin{proof}
We first assume that $f\in H^\vz\cap L^2$.
By the well known Calder\'on reproducing formula (see also \cite[Theorem 2.14]{lfy15}),
we know that there exists a sequence of multiples of $(\vz,\,q,\,s)$-atoms $\{b_j\}_{j\in\zz_+}$
associated with balls $\{x_j+B_{r_j}\}_{j\in\zz_+}$ such that
\begin{align}\label{a2}
f=\lim_{k\rightarrow\fz}\sum_{j=1}^k b_j
=:\lim_{k\rightarrow\fz}f_k \ {\rm{in}} \ \cs' \ {\rm{and \ also \ in}} \ L^2.
\end{align}
From the assumption that the linear or positive sublinear operator $T$
is bounded on $L^2$, Lemma \ref{szxz} and \eqref{a2}, it follows that
\begin{align*}
\lim_{k\rightarrow\fz}\lf\|T(f)-T\lf(f_k\r)\r\|_{L^2}
\le \lim_{k\rightarrow\fz}\lf\|T\lf(f-f_k\r)\r\|_{L^2}\ls \lim_{k\rightarrow\fz}\lf\|f-f_k\r\|_{L^2}\sim0,
\end{align*}
which implies that
\begin{align}\label{z14}
T(f)=\lim_{k\rightarrow\fz}T\lf(f_k\r)
\le\lim_{k\rightarrow\fz}\sum_{j=1}^kT\lf(b_j\r)
=\sum_{j=1}^\fz T\lf(b_j\r) \ {\rm{almost \ everywhere}}.
\end{align}
By this, Lemma \ref{ckj} and \eqref{a1} with taking $\lambda=\Lz_q(\{b_j\}_j)$, we obtain 
\begin{align*}
\int_{\rn}\vz\lf(x,\,\frac{|T(f)(x)|}{\Lz_q(\{b_j\}_j)}\r)\,dx
&\ls \sum_{j=1}^{\fz}\int_{\rn}\vz\lf(x,\,\frac{|T(b_j)(x)|}{\Lz_q(\{b_j\}_j)}\r)\,dx \\
&\ls \sum_{j=1}^{\fz}\vz\lf(x_j+B_{r_j},\,
\frac{\|b_j\|_{L^q_\vz(x_j+B_{r_j})}}{\Lz_q(\{b_j\}_j)}\r)\ls1,
\end{align*}
which, together with Lemma \ref{ky4.3}, further implies that
$$\|T(f)\|_{L^\vz}\ls \Lz_q(\{b_j\}_j).$$
Taking infimum for all admissible decompositions of $f$ as above and using Lemma \ref{yzfj}, we obtain that, for any $f\in H^\vz\cap L^2$,
\begin{align}\label{a3}
\|T(f)\|_{L^\vz} \ls \|f\|_{H^{\vz,\,q,\,s}_{{\rm{at}}}} \sim \|f\|_{H^\vz}.
\end{align}

Next, suppose $f\in{H^\vz}$. By Lemma \ref{cm}, we know that there exists a sequence of
$\{f_j\}_{j\in\zz_+}\subset H^\vz\cap L^2$ such that
$f_j\rightarrow f$ as $j\rightarrow\fz$ in ${H^\vz}$.
Therefore, $\{f_j\}_{j\in\zz_+}$ is a Cauchy sequence in $H^\vz$.
From this, Lemma \ref{szxz} and \eqref{a3}, we deduce that, for any $j,\,k\in\zz_+$,
$$\lf\|T(f_j)-T(f_k)\r\|_{L^\vz} \leq \lf\|T(f_j-f_k)\r\|_{L^\vz}\ls \lf\|f_j-f_k\r\|_{H^\vz}.$$
Thus, by this, we know that
$\{T(f_j)\}_{j\in\zz_+}$ is a Cauchy sequence in $L^\vz$.
Applying Theorem \ref{lwbx}, we conclude that there exists some $g\in L^\vz$
such that $T(f_j)\rightarrow g$ as $j\rightarrow\fz$ in $L^\vz$.
Consequently, define $T(f):=g$. Below, we claim that $T(f)$ is well defined. Indeed,
for any other sequence $\{f'_j\}_{j\in\zz_+}\subset H^\vz\cap L^2$ satisfying
$f'_j\rightarrow f$ as $j\rightarrow\fz$ in ${H^\vz}$, by Lemma \ref{szxz} and \eqref{a3},
we have
\begin{align*}
\lf\|T(f'_j)-T(f)\r\|_{L^\vz}
&\ls \lf\|T(f'_j)-T(f_j)\r\|_{L^\vz}+\lf\|T(f_j)-g\r\|_{L^\vz} \\
&\ls \lf\|f'_j-f_j\r\|_{H^\vz}+\lf\|T(f_j)-g\r\|_{L^\vz} \\
&\ls \lf\|f'_j-f\r\|_{H^\vz}+\lf\|f-f_j\r\|_{H^\vz}+\lf\|T(f_j)-g\r\|_{L^\vz}\rightarrow0 \ {\rm{as}} \ j\rightarrow\fz,
\end{align*}
which is wished. From this, Lemma \ref{nfbds} and \eqref{a3}, it follows that, for any $f\in{H^\vz}$,
$$\|T(f)\|_{L^\vz}=\|g\|_{L^\vz}=\lim_{j\rightarrow\fz}\|T(f_j)\|_{L^\vz}
\ls \lim_{j\rightarrow\fz}\|f_j\|_{H^\vz}\thicksim\|f\|_{H^\vz}.$$
This finishes the proof of Theorem \ref{yt}.
\end{proof}

To show the boundedness criterion of operators from ${H^\vz}$ to $WL^\vz$,
we need the following superposition principle of weak type estimates.
\begin{lemma}\label{dj}{\rm{\cite[Lemma 7.13]{bckyy13}}}
Let $\vz$ be a growth function as in Definition \ref{d2.3}
satisfying $I(\vz)\in(0,\,1)$, where $I(\vz)$ is as in \eqref{e2.1.1}.
Assume that $\{f_j\}_{j\in\zz+}$ is a sequence of measurable functions such that, for some $\lz\in(0,\,\fz)$,
$$\sum_{j\in\zz+}\sup_{\az\in(0,\,\fz)}\vz\lf(\{|f_j|>\az\},\,\frac{\az}{\lz}\r)<\fz.$$
Then there exists a positive constant $C$, depending only on $\vz$, such that, for any $\eta\in(0,\,\fz)$,
$$\vz\lf(\lf\{ \sum_{j\in\zz+}|f_j|>\eta \r\},\,\frac{\eta}{\lz}\r)
\le C\sum_{j\in\zz+}\sup_{\az\in(0,\,\fz)}\vz\lf(\{|f_j|>\az\},\,\frac{\az}{\lz}\r).$$
\end{lemma}

By an argument similar to that used in the proof of \cite[Lemma 4.3]{k14}, we easily obtain the following lemma, the details being omitted.
\begin{lemma}\label{fs1}
Let $\vz$ be a growth function as in Definition \ref{d2.3}.
For a given positive constant $\tilde{C}$,
there exists a positive constant $C$ such that, for any $\lz\in(0,\,\fz)$,
$$\sup_{\alpha\in(0,\,\fz)}\vz\lf(\{|f|>\alpha\},\,\frac{\alpha}{\lambda}\r)\le\tilde{C} \ implies \ that \ \|f\|_{WL^\vz}\le C \lambda.$$
\end{lemma}

The following theorem gives a boundedness criterion of operators from ${H^\vz}$ to $WL^\vz$.

\begin{theorem}\label{yt2}
Let $\vz$ be a growth function as in Definition \ref{d2.3} satisfying
$I(\vz)\in(0,\,1)$, where $I(\vz)$ is as in \eqref{e2.1.1}. Suppose that a linear or a positive sublinear operator $T$ is bounded on $L^2$.
If there exists a positive constant $C$ such that,
for any $\lz\in(0,\,\fz)$ and multiple of a $(\vz,\,q,\,s)$-atom $b$ associated with some ball $B\subset\rn$,
\begin{align}\label{z11}
 \sup_{\alpha\in(0,\,\fz)}\vz\lf(\lf\{|T(b)|>\alpha\r\},\,\frac{\alpha}{\lambda}\r)
 \le C\vz\lf(B,\,\frac{{\|b\|_{L^q_\vz(B)}}}{\lambda}\r),
\end{align}
then $T$ extends uniquely to a bounded operator from ${H^\vz}$ to $WL^\vz$.
\end{theorem}
\begin{proof}
Since the proof of Theorem \ref{yt2} is similar to that of Theorem \ref{yt}, we use the same notation as in the proof of Theorem \ref{yt}.
Here we just give out the necessary modifications.

By \eqref{z14}, Lemma \ref{dj} and \eqref{z11} with taking $\lambda=\Lz_q(\{b_j\}_j)$, we obtain that, for any $\alpha\in(0,\,\fz)$,
\begin{align*}
\vz\lf(\{|T(f)|>\alpha\},\,\frac{\alpha}{\Lz_q(\{b_j\}_j)}\r)
&\le \vz\lf(\lf\{\sum_{j=1}^\fz |T\lf(b_j\r)|>\alpha\r\},\,\frac{\alpha}{\Lz_q(\{b_j\}_j)}\r) \\
&\ls \sum_{j=1}^\fz\sup_{\alpha\in(0,\,\fz)}\vz\lf(\lf\{|T\lf(b_j\r)|>\alpha\r\},\,\frac{\alpha}{\Lz_q(\{b_j\}_j)}\r)  \\
&\ls \sum_{j=1}^{\fz}\vz\lf(x_j+B_{r_j},\,
\frac{\|b_j\|_{L^q_\vz(x_j+B_{r_j})}}{\Lz_q(\{b_j\}_j)}\r)\ls1,
\end{align*}
which, together with Lemma \ref{fs1}, implies that
$$\|T(f)\|_{WL^\vz}\ls \Lz_q(\{b_j\}_j).$$
Taking infimum for all admissible decompositions of $f$ as above and using Lemma \ref{yzfj}, we obtain that, for any $f\in H^\vz\cap L^2$,
\begin{align}\label{z13}
\|T(f)\|_{WL^\vz} \ls \|f\|_{H^{\vz,\,q,\,s}_{{\rm{at}}}} \sim \|f\|_{H^\vz}.
\end{align}

Next, suppose $f\in{H^\vz}$. By Lemma \ref{cm}, we know that there exists a sequence of
$\{f_j\}_{j\in\zz_+}\subset H^\vz\cap L^2$ such that
$f_j\rightarrow f$ as $j\rightarrow\fz$ in ${H^\vz}$.
Therefore, $\{f_j\}_{j\in\zz_+}$ is a Cauchy sequence in $H^\vz$.
From this, Lemma \ref{szxz} and \eqref{z13}, we deduce that, for any $j,\,k\in\zz_+$,
$$\lf\|T(f_j)-T(f_k)\r\|_{{WL}^\vz} \leq \lf\|T(f_j-f_k)\r\|_{{WL}^\vz}\ls \lf\|f_j-f_k\r\|_{H^\vz}.$$
Thus, by this, we know that
$\{T(f_j)\}_{j\in\zz_+}$ is a Cauchy sequence in ${WL}^\vz$.
Applying Theorem \ref{wlwbx}, we conclude that there exists some $g\in {WL}^\vz$
such that $T(f_j)\rightarrow g$ as $j\rightarrow\fz$ in ${WL}^\vz$.
Consequently, define $T(f):=g$. Below, we claim that $T(f)$ is well defined. Indeed,
for any other sequence $\{f'_j\}_{j\in\zz_+}\subset H^\vz\cap L^2$ satisfying
$f'_j\rightarrow f$ as $j\rightarrow\fz$ in ${H^\vz}$, by Lemma \ref{szxz} and \eqref{z13},
we have
\begin{align*}
\lf\|T(f'_j)-T(f)\r\|_{WL^\vz}
&\ls \lf\|T(f'_j)-T(f_j)\r\|_{WL^\vz}+\lf\|T(f_j)-g\r\|_{WL^\vz} \\
&\ls \lf\|f'_j-f_j\r\|_{H^\vz}+\lf\|T(f_j)-g\r\|_{WL^\vz} \\
&\ls \lf\|f'_j-f\r\|_{H^\vz}+\lf\|f-f_j\r\|_{H^\vz}+\lf\|T(f_j)-g\r\|_{WL^\vz}\rightarrow0 \ {\rm{as}} \ j\rightarrow\fz,
\end{align*}
which is wished. From this, Lemma \ref{nfbds} and \eqref{z13}, it follows that, for any $f\in{H^\vz}$,
$$\|T(f)\|_{{WL}^\vz}=\|g\|_{{WL}^\vz}=\lim_{j\rightarrow\fz}\|T(f_j)\|_{{WL}^\vz}
\ls \lim_{j\rightarrow\fz}\|f_j\|_{H^\vz}\sim\|f\|_{H^\vz}.$$
This finishes the proof of Theorem \ref{yt2}.
\end{proof}

\section{Boundedness of parametric Marcinkiewicz integrals \label{s4}}
In this section, we first recall the notion concerning the {\it{$L^q$-Dini type condition of order $\alpha$}} with $q\in[1,\,\fz]$ and $\az\in(0,\,1]$,
and then obtain the boundedness of $\mu^\rho_\boz$ from $H^\vz$ to $L^\vz$ or from $H^\vz$ to $WL^\varphi$.

Here and hereafter, we always assume that $\Omega$ is homogeneous of degree zero and satisfies \eqref{e1.1}.

Recall that, for any $q\in[1,\,\fz)$ and $\alpha\in(0,\,1]$, a function $\Omega\in L^q(S^{n-1})$ is said to satisfy the
{\it{$L^q$-Dini type condition of order $\alpha$}} (when $\alpha=0$, it is called the {\it{$L^q$-Dini condition}}),
if $$\int_0^1\frac{\omega_q(\delta)}{\delta^{1+\alpha}}\,d\delta<\fz,$$
where ${\omega_q(\delta)}$ is the integral modulus of continuity of order $q$ of $\Omega$
defined by setting, for any $\delta\in(0,\,1]$,
$${\omega_q(\delta)}:=\sup_{\|\gz\|<\delta}
\lf(\int_{S^{n-1}}|\Omega(\gz x')-\Omega(x')|^q\,d\sigma(x')\r)^{1/q}$$
and $\gz$ denotes a rotation on $S^{n-1}$ with $\|\gz\|:=\sup_{y'\in S^{n-1}}|\gz y'-y'|$.
For any $\alpha,\,\beta\in(0,\,1]$ with $\beta<\alpha$,
it is easy to see that if $\Omega$ satisfies the $L^q$-Dini type condition of order $\alpha$,
then it also satisfies the $L^q$-Dini type condition of order $\beta$.
We thus denote by ${\rm{Din}}^q_\alpha(S^{n-1})$ the class of all functions which satisfy the
$L^q$-Dini type conditions of all orders $\beta<\alpha$. For any $\alpha\in(0,\,1]$, we define
$${\rm{Din}}^\fz_\alpha(S^{n-1}):=\bigcap_{q\ge1}{\rm{Din}}^q_\alpha(S^{n-1}).$$
See \cite[pp.\,89-90]{ll07} for more properties of ${\rm{Din}}^q_\alpha(S^{n-1})$
with $q\in[1,\,\fz]$ and $\alpha\in(0,\,1]$.

The main results of this section are as follows.

\begin{theorem}\label{dingli.1}
Let $\rho\in(0,\,\infty)$, $\alpha\in(0,\,1]$, $\beta:=\min\{\alpha,\,1/2\}$ and $\vz$ be a growth function as in Definition \ref{d2.3} with $p\in({n}/{(n+\beta)},\,1)$ therein.
Suppose that $\Omega\in L^q(S^{n-1})\cap {\rm{Din}}^1_{\alpha}(S^{n-1})$ with $q\in(1,\,\fz]$.
If $q$ and $\varphi$ satisfy one of the following conditions:
\begin{enumerate}
\item[\rm{(i)}] $q\in(1,\,1/p]$ and $\vz^{q'}\in \mathbb{A}_{\frac{p\beta}{n(1-p)}}$;
\item[\rm{(ii)}] $q\in(1/p,\,\fz]$ and $\vz^{{1}/{(1-p)}}\in \mathbb{A}_{\frac{p\beta}{n(1-p)}}$,
\end{enumerate}
then there exists a positive constant $C$ independent of $f$ such that
$$\lf\|\mu_{\Omega}^{\rho}(f)\r\|_{L^\vz} \leq C\|f\|_{H^\vz}.$$
\end{theorem}

\begin{theorem}\label{dingli.2}
Let $\rho\in(0,\,\infty)$, $\alpha\in(0,\,1]$, $\beta:=\min\{\alpha,\,1/2\}$ and $\vz$ be a growth function as in Definition \ref{d2.3} with $p\in({n}/{(n+\beta)},\,1]$ therein.
Suppose that $\Omega\in {\rm{Din}}^q_{\alpha}(S^{n-1})$ with $q\in(1,\,\fz)$.
If $\vz^{q'}\in\aa_{(p+\frac{p\beta}{n}-\frac1q)q'}$,
then there exists a positive constant $C$ independent of $f$ such that
$$\lf\|\mu_{\Omega}^{\rho}(f)\r\|_{L^\vz} \leq C\|f\|_{H^\vz}.$$
\end{theorem}

\begin{corollary}\label{tuilun.1}
Let $\rho\in(0,\,\infty)$, $\alpha\in(0,\,1]$, $\beta:=\min\{\alpha,\,1/2\}$
and $\vz$ be a growth function as in Definition \ref{d2.3} with $p\in({n}/{(n+\beta)},\,1]$ therein.
Suppose that $\Omega\in {\rm{Din}}^\fz_{\alpha}(S^{n-1})$.
If $\vz\in\aa_{p(1+\frac{\beta}{n})}$,
then there exists a positive constant $C$ independent of $f$ such that
$$\lf\|\mu_{\Omega}^{\rho}(f)\r\|_{L^\vz} \leq C\|f\|_{H^\vz}.$$
\end{corollary}

\begin{theorem}\label{dingli.3}
Let $\rho\in(0,\,\infty)$ and $\vz$ be a growth function as in Definition \ref{d2.3} with $p:=1$ therein.
For a given positive constant $\tilde{C}$, suppose $\Omega\in{L}^{q}(S^{n-1})$ with $q\in(1,\,\infty)$
such that, for any $y\neq{\mathbf{0}}$, $h\in\rn$ and $t\in[0,\,\infty)$,
\begin{align*}
\int_{|x|\geq2|y|}\lf|\frac{\Omega(x-y)}{|x-y|^{n}}-\frac{\Omega(x)}{|x|^{n}}\r| \vz(x+h,\,t)\,dx
\leq \tilde{C}\vz(x+h,\,t).
\end{align*}
If $\varphi^{q'}\in\aa_{1}$,
then there exists a positive constant $C$ independent of $f$ such that
$$\lf\|\mu_{\Omega}^{\rho}(f)\r\|_{L^\vz} \leq C\|f\|_{H^\vz}.$$
\end{theorem}

\begin{theorem}\label{dingli.4}
Let $\rho\in(0,\,\infty)$, $\alpha\in(0,\,1]$, $\beta:=\min\{\alpha,\,1/2\}$ and
$\vz$ be a growth function as in Definition \ref{d2.3} with $p:={n}/{(n+\beta)}$ therein, and $I(\vz)\in(0,\,1)$, where $I(\vz)$ is as in \eqref{e2.1.1}.
Suppose that $\Omega\in{\rm{Lip}}_\alpha(S^{n-1})$.
If $\varphi\in\aa_{1}$, then there exists a positive constant $C$ independent of $f$ such that
$$\lf\|\mu_{\Omega}^{\rho}(f)\r\|_{WL^\vz} \leq C\|f\|_{H^\vz}.$$
\end{theorem}

\begin{remark}\label{r3}
\begin{enumerate}
\item[\rm{(i)}]
It is worth noting that Corollary \ref{tuilun.1} can be regarded as the limit
case of Theorem \ref{dingli.2} by letting $q\rightarrow\fz$.

\item[\rm{(ii)}]
Theorem \ref{dingli.1}, Theorem \ref{dingli.2} and Corollary \ref{tuilun.1} jointly answer the question:
when $\Omega\in{\rm{Din}}^q_{\alpha}(S^{n-1})$ with $q=1$, $q\in(1,\,\infty)$ or $q=\infty$,
respectively, what kind of additional conditions on growth function $\varphi$ and $\Omega$
can deduce the boundedness of $\mu^\rho_\Omega$ from $H^\vz$ to $L^\vz$?

\item[\rm{(iii)}]
When $\rho:=1$, Theorem \ref{dingli.1}, Theorem \ref{dingli.2} and Corollary \ref{tuilun.1}
are reduced to \cite[Theorem 2.4, Theorem 2.5 and Corollary 2.6]{lll17}, respectively.

\item[\rm{(iv)}]
Let $\omega$ be a classic Muckenhoupt weight and $\phi$ an Orlicz function.

\quad(a) When $\vz(x,\,t):=\omega(x)\phi(t)$ for all $(x,\,t)\in\rn\times[0,\,\fz)$,
we have $H^\vz=H^\phi_\omega$. In this case,
Theorem \ref{dingli.1}, Theorem \ref{dingli.2}, Corollary \ref{tuilun.1} and Theorem \ref{dingli.3} hold true for weighted Orlicz Hardy space. Even when $\varphi(x,\,t):=\phi(t)$,
the above results are also new.

\quad(b) When $\vz(x,\,t):=\omega(x)\,t^p$  for all $(x,\,t)\in\rn\times[0,\,\fz)$,
we have $H^\vz=H^p_\omega$. In this case, if $\rho:=1$, Theorem \ref{dingli.1}, Theorem \ref{dingli.2}, Corollary \ref{tuilun.1}
and Theorem \ref{dingli.3} are reduced to \cite[Theorem 1.4, Theorem 1.5, Corollary 1.7 and Theorem 1.8]{ll07}, respectively.

\quad(c) When $\vz(x,\,t):=\omega(x)\,t^p$  for all $(x,\,t)\in\rn\times[0,\,\fz)$,
we have $H^\vz=H^p_\omega$.
In this case, the assumptions of $\rho$ and $\Omega$
in Corollary \ref{tuilun.1} are weaker than that in \cite[Theorem 1.1]{w16}.
Precisely, in \cite[Theorem 1.1]{w16}, $\rho\in(0,\,n)$ and $\Omega\in{\rm{Lip}}_\alpha(S^{n-1})$,
however, in our case, $\rho\in(0,\,\fz)$ and $\Omega$ satisfies some weaker smoothness conditions, i.e., $\Omega\in{\rm{Din}}^\fz_{\alpha}(S^{n-1})$.

\quad(d) When $\vz(x,\,t):=\omega(x)\,t^p$  for all $(x,\,t)\in\rn\times[0,\,\fz)$, we have
$H^\vz=H^p_\omega$. In this case, if $\rho$ is restricted to $(0,\,n)$, Theorem \ref{dingli.4} is reduced to \cite[Theorem 1.2]{w16}.
\end{enumerate}
\end{remark}




To show main results, let us begin with some lemmas.
Since $\varphi$ satisfies the uniform Muckenhoupt condition, the
proofs of $\rm{(i)}$, $\rm{(ii)}$ and $\rm{(iii)}$ of the following Lemma \ref{quan} are identity to that of
{\rm\cite[Exercises 9.1.3, Theorem 9.2.5 and Corollary 9.2.6]{g09m}}, respectively, the details being omitted.
\begin{lemma}\label{quan}
Let $q\in[1,\,\fz]$. If $\vz\in \aa_q$, then the following statements hold true:
\begin{enumerate}
\item[\rm{(i)}] $\vz^\varepsilon\in \aa_q$ for any $\varepsilon\in(0,\,1]$;
\item[\rm{(ii)}]$\vz^\eta\in \aa_q$ for some $\eta\in(1,\,\fz)$;
\item[\rm{(iii)}]$\vz\in \aa_d$ for some $d\in(1,\,q)$ with $q\neq1$.
\end{enumerate}
\end{lemma}

The following Lemma \ref{lemma.1} is a subtle pointwise estimate
for $\mu^\rho_\Omega(b)$,
where $b$ is a multiple of a $(\vz,\,\fz,\,s)$ atom. And this lemma plays an important role in
the proof of Theorem \ref{dingli.1}.
\begin{lemma}\label{lemma.1}
Let $\rho\in(0,\,\infty)$ and $b$ be a multiple of a $(\vz,\,\fz,\,s)$-atom associated with some ball $B_r$. Then,
for any $x\in B_{2R}\setminus B_{R}$ with $R\in[2r,\,\fz)$,
$$\mu^\rho_\Omega(b)(x)\le \|\Omega\|_{L^1(S^{n-1})}\|b\|_{L^\fz}\frac{1}{\rho}\lf\{\ln{\frac{2R+r}{R-r}}+\frac{[(2R+r)^{\rho}-(R-r)^{\rho}]^2}{2\rho(2R+r)^{2\rho}}\r\}^{1/2}.$$
\end{lemma}
\begin{proof}
The key of the proof is to find a subtle segmentation. From $\supp\,b \subset B_r$, we deduced that,
for any $y\in B_r$ and $x\in B_{2R}\setminus B_{R}$ with $R\in[2r,\,\fz)$,
\begin{align}\label{eee.1}
R-r<|x-y|<2R+r.
\end{align}
Therefore, for any $x\in B_{2R}\setminus B_{R}$ with $R\in[2r,\,\fz)$, write
\begin{align*}
\lf[\mu^\rho_\Omega(b)(x)\r]^2
&=\int_{0}^{\fz}\lf|\int_{|x-y|\leq t}
\frac{\Omega(x-y)}{|x-y|^{n-\rho}}b(y)\,{dy}\r|^2\,\frac{dt}{t^{2\rho+1}} \\
&= \int_0^{R-r}\lf|\int_{|x-y|\leq t}
\frac{\Omega(x-y)}{|x-y|^{n-\rho}}b(y)\,{dy}\r|^2\,\frac{dt}{t^{2\rho+1}}
+\int_{R-r}^{2R+r}\cdot\cdot\cdot+\int_{2R+r}^\fz\cdot\cdot\cdot \\
&=:{\rm{I_1+I_2+I_3}}.
\end{align*}

For ${\rm{I_1}}$, from $t\in(0,\,R-r]$ and \eqref{eee.1}, it follows that
$\{y\in\rn: \ |x-y|\le t\}=\emptyset$ and hence ${\rm{I_1}}=0$.

For ${\rm{I_2}}$, by the spherical coordinates transform
and $\Omega\in L^1(S^{n-1})$ (see \eqref{e1.1}), we obtain
\begin{align*}
{\rm{I_2}}
&\le\|b\|_{L^\fz}^2 \int_{R-r}^{2R+r}\lf(\int_{S^{n-1}}\int_0^t
\frac{|\Omega(z')|}{u^{n-\rho}}u^{n-1}\,{du}\,{d\sigma(z')}\r)^2\,\frac{dt}{t^{2\rho+1}} \\
&=\|\Omega\|^2_{L^1(S^{n-1})} \|b\|_{L^\fz}^2 \frac{1}{{\rho}^2}\int_{R-r}^{2R+r}\frac1t\,dt
=\|\Omega\|^2_{L^1(S^{n-1})}\|b\|_{L^\fz}^2\frac{1}{{\rho}^2} \ln{\frac{2R+r}{R-r}}.
\end{align*}

For ${\rm{I_3}}$, by \eqref{eee.1}, the spherical coordinates transform and
$\Omega\in L^1(S^{n-1})$ (see \eqref{e1.1}), we have
\begin{align*}
{\rm{I_3}}
&\leq \|b\|_{L^\fz}^2 \int_{2R+r}^\fz\lf(\int_{B_{2R+r}\setminus B_{R-r}}
\frac{|\Omega(z)|}{|z|^{n-\rho}}\,{dz}\r)^2\,\frac{dt}{t^{2\rho+1}} \\
&= \|b\|_{L^\fz}^2 \int_{2R+r}^\fz\lf(\int_{S^{n-1}}\int_{R-r}^{2R+r}
\frac{|\Omega(z')|}{u^{n-\rho}}u^{n-1}\,{du}\,{d\sigma(z')}\r)^2\,\frac{dt}{t^{2\rho+1}} \\
&=\|\Omega\|^2_{L^1(S^{n-1})}\|b\|_{L^\fz}^2\frac{1}{{\rho}^2} [(2R+r)^\rho-(R-r)^\rho]^2 \int_{2R+r}^\fz\frac{1}{t^{2\rho+1}}\,dt\\
&=\|\Omega\|^2_{L^1(S^{n-1})} \|b\|_{L^\fz}^2\frac{1}{{\rho}^2} \frac{[(2R+r)^{\rho}-(R-r)^{\rho}]^2}{2\rho(2R+r)^{2\rho}}.
\end{align*}

Combining the estimates of ${\rm{I_1}}$, ${\rm{I_2}}$ and ${\rm{I_3}}$, we obtain the desired inequality.
This finishes the proof of Lemma \ref{lemma.1}.
\end{proof}

\begin{lemma}\label{bcd}{\rm{\cite[Lemma 4.5]{k14}}}
Let $\vz\in\mathbb{A}_q$ with $q\in[1,\,\fz)$.
Then there exists a positive constant $C$ such that,
for any ball $B\subset\rn$, $\lambda\in(1,\,\fz)$ and $t\in(0,\,\fz)$,
$$\vz(\lambda B,\,t)\le C{\lambda}^{nq}\vz(B,\,t).$$
\end{lemma}

Since $\varphi$ satisfies the uniform Muckenhoupt condition, the proof of Lemma \ref{fh} is identity
to that of {\rm\cite[Corollary 6.2]{sw85}}, the details being omitted.

\begin{lemma}\label{fh}
Let $d\in(1,\,\fz)$. Then, $\vz^d\in\mathbb{A}_\fz$ if and only if $\vz\in{\mathbb{RH}_d}$.
\end{lemma}

The proof of the following Lemma \ref{L3.6} is motivated by \cite[Lemma 5]{kw79}.

\begin{lemma}\label{L3.6}
Suppose that $\rho\in(0,\,\infty)$, $q\in[1,\,\infty)$ and $\Omega$ satisfies the $L^q$-Dini condition. Then there exists a positive constant $C$ such that, for any $R\in(0,\,\infty)$ and $y\in B_{R/2}$,
\begin{align*}
\lf(\int_{R\leq|x|<2R}\lf|\frac{\Omega(x-y)}{|x-y|^{n-\rho}}-\frac{\Omega(x)}{|x|^{n-\rho}}\r|^q dx\r)^{1/q}
\leq CR^{n/q-(n-\rho)}\lf({\frac{|y|}{R}}+
\int_{|y|/2R}^{|y|/R}\frac{\omega_q(\delta)}{\delta}d\delta\r).
\end{align*}
\end{lemma}
\begin{proof}
Noticing that $|x|\geq R$ and $|y|<R/2$, we have $|x-y|\sim|x|$.
From this and the mean value theorem, it follows that
\begin{align*}
\lf|\frac{\Omega(x-y)}{|x-y|^{n-\rho}}-\frac{\Omega(x)}{|x|^{n-\rho}}\r|
&\leq \lf|\frac{\Omega(x)}{|x|^{n-\rho}}-\frac{\Omega(x)}{|x-y|^{n-\rho}}\r|+
\lf|\frac{\Omega(x)}{|x-y|^{n-\rho}}-\frac{\Omega(x-y)}{|x-y|^{n-\rho}}\r|\\
&\leq C\lf(|\Omega(x)|\frac{|y|}{|x|^{n-\rho+1}}+\frac{|\Omega(x-y)-\Omega(x)|}{|x|^{n-\rho}}\r).
\end{align*}
We then write
\begin{align*}
& \lf(\int_{R\leq|x|<2R}\lf|\frac{\Omega(x-y)}{|x-y|^{n-\rho}}-\frac{\Omega(x)}{|x|^{n-\rho}}\r|^q dx\r)^{1/q} \\
&\hs \leq C\lf(\int_{R\leq|x|<2R}|\Omega(x)|^q\frac{|y|^q}{|x|^{(n-\rho+1)q}}dx\r)^{1/q}
+C\lf(\int_{R\leq|x|<2R}\frac{|\Omega(x-y)-\Omega(x)|^q}{|x|^{(n-\rho)q}}dx\r)^{1/q}\\
&\hs =:C(\mathrm{I_1+I_2}).
\end{align*}

For $\mathrm{I_1}$, by the spherical coordinates transform and $\Omega\in{L}^{q}(S^{n-1})$, we know that, for any $y\in B_{R/2}$,
\begin{align*}
&\mathrm{I_1}=|y|\lf(\int_{S^{n-1}}\int_{R}^{2R}|\Omega(x')|^q\frac{r^{n-1}}{r^{(n-\rho+1)q}}drd\sigma(x')\r)^{1/q} \\
&\hs\sim |y|\lf(\int_{R}^{2R}r^{n-1-(n-\rho+1)q} dr\r)^{1/q}
\sim R^{n/q-(n-\rho)}\lf(\frac{|y|}{R}\r).
\end{align*}

For ${\mathrm{I_2}}$, from the spherical coordinates transform and Fubini's theorem,
it follows that, for any $y\in B_{R/2}$,
\begin{align*}
& \mathrm{I_2}=\lf(\int_{S^{n-1}}\int_{R}^{2R}\frac{|\Omega(rx'-y)-\Omega(rx')|^q}{r^{(n-\rho)q}}r^{n-1}drd\sigma(x')\r)^{1/q} \\
&\hs \sim R^{n/q-(n-\rho)}\lf[\int_{R}^{2R}\lf(\int_{S^{n-1}}\lf|\Omega\lf(\frac{x'-\alpha}{|x'-\alpha|}\r)-\Omega(x')\r|^q
d\sigma(x')\r)\frac{dr}{r}\r]^{1/q},
\end{align*}
where $\alpha:=y/r$. Proceeding as in the proof of \cite[Lemma 5]{kw79}, $\mathrm{I_2}$ is bounded by a positive constant times
$$R^{n/q-(n-\rho)}\lf(\int_{|y|/2R}^{|y|/R}\omega_{q}(\delta)\frac{d\delta}{\delta}\r).$$

Combining the estimates of ${\rm{I_1}}$ and ${\rm{I_2}}$, we obtain the desired inequality.
This finishes the proof of Lemma \ref{L3.6}.
\end{proof}

The following Lemma \ref{lin} extends \cite[Lemma 4.4]{ll07} from non-parametric case to the parametric case.

\begin{lemma}\label{lin}
For $\alpha\in(0,\,1]$ and $q\in[1,\,\fz)$, suppose that $\Omega$
satisfies the $L^q$-Dini type condition of order $\alpha$. Let $\rho\in(0,\,\infty)$, $\beta:=\min\{\alpha,\,1/2\}$ and $b$ be a multiple of a  $(\varphi,\infty,s)$-atom associated with some ball $B_r$.
\begin{enumerate}
\item[\rm{(i)}]
If $q=1$,
then there exists a positive constant $C$ independent of $b$ such that, for any $R\in[2r,\,\fz)$,
$$\int_{B_{2R}\setminus B_{R}}\mu^\rho_\Omega(b)(x)\,dx
\leq C\|b\|_{L^\fz}R^n\lf(\frac{r}{R}\r)^{n+\beta}.$$
\item[\rm{(ii)}]
If $q\in(1,\,\fz)$ and, for any $(x,\,t)\in\rn\times[0,\,\fz)$, $\vz(x,\,t)\ge0$,
then there exists a positive constant $C$ independent of $b$ such that,
for any $R\in[2r,\,\fz)$ and $t\in[0,\,\fz)$,
$$\int_{B_{2R}\setminus B_{R}}\mu^\rho_\Omega(b)(x)\vz(x,\,t)\,dx
\leq C\|b\|_{L^\fz}\lf[\vz^{q'}(B_{2R},\,t)\r]^{1/q'}R^{n/q}\lf(\frac{r}{R}\r)^{n+\beta}.$$
\end{enumerate}
\end{lemma}

\begin{proof}
We only prove for case (ii), since the proof of case (i) is analogous to that of case (ii) and is left to the readers. For any $R\in[2r,\,\fz)$ and $t\in[0,\,\fz)$, write
\begin{align*}
& \int_{R\leq|x|<2R}\mu^\rho_\Omega(b)(x)\vz(x,\,t)\,dx \\
&\hs \leq\int_{R\leq|x|<2R}\lf(\int_{0}^{|x|+r}\lf|\int_{|x-y|\leq t}
\frac{\Omega(x-y)}{|x-y|^{n-\rho}}b(y)\,{dy}\r|^2\,\frac{dt}{t^{2\rho+1}}\r)^{1/2}\vz(x,\,t)\,dx \\
&\hs\hs+\int_{R\leq|x|<2R}\lf(\int_{|x|+r}^{\infty}\lf|\int_{|x-y|\leq t}
\frac{\Omega(x-y)}{|x-y|^{n-\rho}}b(y)\,{dy}\r|^2\,\frac{dt}{t^{2\rho+1}}\r)^{1/2}\vz(x,\,t)\,dx=:\mathrm{I_1+I_2}.
\end{align*}

For $\mathrm{I_1}$, noticing that $y\in B_r$ and $x\in B_{2R}\setminus B_{R}$ with $R\in[2r,\,\fz)$, we know that
\begin{align}\label{eee.2}
   |x-y|\sim|x|\sim|x|+r
\end{align}
and
\begin{align}\label{eee.3}
   R/2<|x-y|<5R/2.
\end{align}
From \eqref{eee.2} and the mean value theorem, it follows that, for any $y\in B_r$ and $x\in B_{2R}\setminus B_{R}$ with $R\in[2r,\,\fz)$,
$$\lf|\frac{1}{|x-y|^{2\rho}}-\frac{1}{(|x|+r)^{2\rho}}\r|\ls\frac{r}{|x-y|^{2\rho+1}}.$$
By Minkowski's inequality for integrals, the above inequality and Fubini's theorem, we obtain that, for any $R\in [2r,\,\infty)$ and $t\in[0,\,\fz)$,
\begin{align*}
{\mathrm{I_1}}
&\leq \int_{R\leq|x|<2R}\lf[\int_{B_{r}}\lf|\frac{\Omega(x-y)}{|x-y|^{n-\rho}}b(y)\r|
\lf(\int_{|x-y|}^{|x|+r}\frac{dt}{t^{2\rho+1}}\r)^{1/2}\,dy\r]\vz(x,\,t)\,dx \\
&\ls \|b\|_{L^\fz} \int_{R\leq|x|<2R}\lf[\int_{B_{r}}\frac{\lf|\Omega(x-y)\r|}{|x-y|^{n-\rho}}
\lf|\frac{1}{|x-y|^{2\rho}}-\frac{1}{(|x|+r)^{2\rho}}\r|^{1/2}\,dy\r]\vz(x,\,t)\,dx \\
&\ls \|b\|_{L^\fz}r^{1/2}\int_{B_{r}}\lf(\int_{R\leq|x|<2R}\frac{|\Omega(x-y)|}{|x-y|^{n+{1/2}}}\vz(x,\,t)\,dx\r)dy.
\end{align*}
On the other hand, from H\"{o}lder's inequality, \eqref{eee.3}, the spherical coordinates transform and $\Omega\in{L}^{q}(S^{n-1})$, we deduced that, for any $y\in {B_{r}}$, $R\in [2r,\,\infty)$ and $t\in[0,\,\fz)$,
\begin{align*}
&\int_{R\leq|x|<2R}\frac{|\Omega(x-y)|}{|x-y|^{n+{1/2}}}\vz(x,\,t)\,dx\\
&\hs\leq \lf(\int_{R\leq|x|<2R}\frac{|\vz(x,\,t)|^{q'}}{|x-y|^{n+{1/2}}}\,dx\r)^{1/q'}
\lf(\int_{R\leq|x|<2R}\frac{|\Omega(x-y)|^q}{|x-y|^{n+{1/2}}}dx\r)^{1/q}\\
&\hs\ls\lf[\vz^{q'}(B_{2R},\,t)\r]^{1/q'} R^{(-n-1/2)/q'}
\lf(\int_{R/2<|z|<5R/2}\frac{|\Omega(z)|^q}{|z|^{n+{1/2}}}dz\r)^{1/q}\\
&\hs\sim \lf[\vz^{q'}(B_{2R},\,t)\r]^{1/q'}R^{(-n-1/2)/q'}\lf(R^{-n-1/2}
\int_{S^{n-1}}\int_{0}^{{5R}/2}\lf|\Omega(z')\r|^q u^{n-1} du d\sigma(z')\r)^{1/q}\\
&\hs\sim \lf[\vz^{q'}(B_{2R},\,t)\r]^{1/q'}R^{-n/q'-1/2}.
\end{align*}
Substituting the above inequality into ${\mathrm{I_1}}$ and using the assumption that $\beta=\min\{\alpha,\,1/2\}$, we know that, for any $R\in [2r,\,\infty)$ and $t\in[0,\,\fz)$,
$${\mathrm{I_1}}\ls  \|b\|_{L^\fz}\lf[\vz^{q'}(B_{2R},\,t)\r]^{1/q'}R^{n/q}\lf(\frac{r}{R}\r)^{n+1/2}
\ls  \|b\|_{L^\fz}\lf[\vz^{q'}(B_{2R},\,t)\r]^{1/q'}R^{n/q}\lf(\frac{r}{R}\r)^{n+\beta}.$$

For ${\mathrm{I_2}}$, noticing that for $t>|x|+r$, it is easy to see that $B_r\subset \lf\{y\in\mathbb{R}^n:|x-y|\leq t\r\}$. From this, vanishing moments of $b$, Minkowski's inequality for integrals and Fubini's theorem, it follows that, for any $R\in [2r,\,\infty)$ and $t\in[0,\,\fz)$,
\begin{align*}
{\mathrm{I_2}}
&= \int_{R\leq|x|<2R}\lf(\int_{|x|+r}^{\infty}\lf|\int_{|x-y|\leq t}
\lf(\frac{\Omega(x-y)}{|x-y|^{n-\rho}}-\frac{\Omega(x)}{|x|^{n-\rho}}\r)b(y)\,{dy}\r|^2\,\frac{dt}{t^{2\rho+1}}\r)^{1/2}\vz(x,\,t)\,dx \\
&\leq \int_{R\leq|x|<2R}\lf[\int_{B_r}\lf|\frac{\Omega(x-y)}{|x-y|^{n-\rho}}-\frac{\Omega(x)}{|x|^{n-\rho}}\r||b(y)|
\lf(\int_{R}^{\infty}\frac{dt}{t^{2\rho+1}}\r)^{1/2} dy\r]\vz(x,\,t)\,dx \\
&\ls\|b\|_{L^\fz}R^{-\rho}\int_{B_r}\lf(\int_{R\leq|x|<2R}\lf|\frac{\Omega(x-y)}{|x-y|^{n-\rho}}-\frac{\Omega(x)}{|x|^{n-\rho}}\r| \vz(x,\,t)\,dx\r)dy.
\end{align*}
On the other hand, from H\"{o}lder's inequality and Lemma \ref{L3.6} (since $\Omega$
satisfies the $L^q$-Dini type condition of order $\alpha$, it also satisfies the $L^q$-Dini condition),
we deduced that, for any $y\in {B_{r}}$, $R\in [2r,\,\infty)$ and $t\in[0,\,\fz)$,
\begin{align*}
&\int_{R\leq|x|<2R}\lf|\frac{\Omega(x-y)}{|x-y|^{n-\rho}}-\frac{\Omega(x)}{|x|^{n-\rho}}\r| \vz(x,\,t)\,dx\\
&\hs\leq \lf(\int_{R\leq|x|<2R}|\vz(x,\,t)|^{q'}dx\r)^{1/q'}
\lf(\int_{R\leq|x|<2R}\lf|\frac{\Omega(x-y)}{|x-y|^{n-\rho}}-\frac{\Omega(x)}{|x|^{n-\rho}}\r|^q dx\r)^{1/q}\\
&\hs\ls \lf[\vz^{q'}(B_{2R},\,t)\r]^{1/q'}R^{n/q-n+\rho}\lf({\frac{|y|}{R}}+
\int_{|y|/2R}^{|y|/R}\frac{\omega_q(\delta)}{\delta}d\delta\r)\\
&\hs\ls \lf[\vz^{q'}(B_{2R},\,t)\r]^{1/q'}R^{n/q-n+\rho}\lf[{\frac{|y|}{R}}+\lf(\frac{|y|}{R}\r)^\alpha
\int_{|y|/2R}^{|y|/R}\frac{\omega_q(\delta)}{\delta^{1+\alpha}}d\delta\r].
\end{align*}
Substituting the above inequality into ${\mathrm{I_2}}$ and using the assumptions that $\Omega$ satisfies the $L^q$-Dini type condition of order $\alpha$, and $\beta=\min\{\alpha,\,1/2\}$, we know that, for any $R\in [2r,\,\infty)$ and $t\in[0,\,\fz)$,
\begin{align*}
{\mathrm{I_2}}
&\ls \|b\|_{L^\fz}\lf[\vz^{q'}(B_{2R},\,t)\r]^{1/q'}R^{n/q-n}\int_{B_{r}}\lf[{\frac{r}{R}}+\lf({\frac{r}{R}}\r)^\alpha
\int_{0}^{1}\frac{\omega_q(\delta)}{\delta^{1+\alpha}}d\delta\r]dy \\
&\ls \|b\|_{L^\fz}\lf[\vz^{q'}(B_{2R},\,t)\r]^{1/q'}R^{n/q}\lf({\frac{r}{R}}\r)^{n}\lf[{\frac{r}{R}}+\lf({\frac{r}{R}}\r)^\alpha\r]  \\
&\ls \|b\|_{L^\fz}\lf[\vz^{q'}(B_{2R},\,t)\r]^{1/q'}R^{n/q}\lf({\frac{r}{R}}\r)^{n+\beta}.
\end{align*}

Combining the estimates of ${\rm{I_1}}$ and ${\rm{I_2}}$, we obtain the desired inequality.
This finishes the proof of Lemma \ref{lin}.
\end{proof}

The following Lemma \ref{lemma.2} shows that $\mu_\Oz$ maps all multiple of an atoms into
uniformly bounded elements of $L^\vz$.

\begin{lemma}\label{lemma.2}
Let $\rho\in(0,\,\infty)$, $\alpha\in(0,\,1]$, $\beta:=\min\{\alpha,\,1/2\}$ and
$p\in({n}/{(n+\beta)},\,1)$.
Suppose that $\Omega\in L^q(S^{n-1})\cap {\rm{Din}}^1_{\alpha}(S^{n-1})$ with $q\in(1,\,\fz]$. If $q$ and $\varphi$ satisfy one of the following conditions:
\begin{enumerate}
\item[\rm{(i)}] $q\in(1,\,1/p]$ and $\vz^{q'}\in \mathbb{A}_{\frac{p\beta}{n(1-p)}}$;
\item[\rm{(ii)}] $q\in(1/p,\,\fz]$ and $\vz^{{1}/{(1-p)}}\in \mathbb{A}_{\frac{p\beta}{n(1-p)}}$,
\end{enumerate}
then there exists a positive constant $C$ such that,
for any $\lz\in(0,\,\fz)$ and  multiple of a $(\vz,\,\fz,\,s)$-atom $b$ associated with some ball $B\subset\rn$,
$$\int_{\rn}\vz\lf(x,\,\frac{\mu^\rho_\Omega(b)(x)}{\lz}\r)\,dx \le C\vz\lf(B,\,\frac{\|b\|_{L^\fz}}{\lz}\r).$$
\end{lemma}

\begin{proof}
We need only consider the case $q\in(1,\,\fz)$, since the case $q=\infty$ can be derived from the case $q=2$.
Indeed, when $q=\infty$, a routine computation gives rise to $2>1/p$.
If Lemma \ref{lemma.2} holds true for $q=2$, by $\Omega\in L^\infty(S^{n-1})\subset L^2(S^{n-1})$, $2>1/p$
and $\vz^{{1}/{(1-p)}}\in \mathbb{A}_{\frac{p\beta}{n(1-p)}}$, we know that Lemma \ref{lemma.2} holds true for $q=\fz$.
We are now turning to the proof of Lemma \ref{lemma.2} under case $q\in(1,\,\fz)$.
Without loss of generality, we may assume $b$ is a multiple of a
$(\vz,\,\fz,\,s)$-atom associated with a ball $B_r$ for some $r\in(0,\,\fz)$.
For the general case,
we refer the readers to the method of proof in \cite[Theorem 1.4]{ll07}.

First, we claim that, in either case (i) or (ii) of Lemma \ref{lemma.2},
there exists some $d\in(1,\,{p\beta}/{[n(1-p)]})$
such that
\begin{align}\label{eee.4}
  \vz^{q'}\in\mathbb{A}_d \ {\rm{and}} \ \vz^{1/(1-p)}\in\mathbb{A}_d.
\end{align}
We only prove \eqref{eee.4} under case (ii) since the proof under case (i) is similar.
By Lemma \ref{quan}(iii) with $\vz^{1/{(1-p)}}\in\aa_{\frac{p\beta}{n(1-p)}}$, we see that
there exists some $d\in(1,\,{p\beta}/{[n(1-p)]})$ such that $\vz^{1/(1-p)}\in\mathbb{A}_d$.
On the other hand, notice that $q'<1/{(1-p)}$, then, by Lemma \ref{quan}(i),
we know $\vz^{q'}\in\mathbb{A}_d$, which is wished.

The next thing to do in the proof is to find a subtle segmentation. For any $j\in\zz_+$, let $E_j:=B_{2^{j+1}r}\setminus B_{2^{j}r}$.
By Lemma \ref{lemma.1}, we know that, for any $x\in E_j$,
\begin{align}\label{eee.x}
  \mu^\rho_\Omega(b)(x) \ls \|b\|_{L^\fz}\frac{1}{\rho}\lf\{\ln{\frac{2^{j+1}+1}{2^j-1}}+
\frac{[(2^{j+1}+1)^{\rho}-(2^{j}-1)^{\rho}]^2}{2\rho(2^{j+1}+1)^{2\rho}}\r\}^{1/2}.
\end{align}
Notice that
$$\lf\{\ln{\frac{2^{j+1}+1}{2^j-1}}+\frac{[(2^{j+1}+1)^{\rho}-(2^{j}-1)^{\rho}]^2}{2\rho(2^{j+1}+1)^{2\rho}}\r\}^{1/2}
\rightarrow\lf[\ln2+\frac{1}{2\rho}\lf(1-\frac1{2^\rho}\r)^2\r]^{1/2}\,\,{\rm{as}}\,\,j\rightarrow\fz,$$
which, together with
$$\sup_{\rho\in(0,\,\infty)}\lf[\ln2+\frac{1}{2\rho}\lf(1-\frac1{2^\rho}\r)^2\r]^{1/2}<1,$$
implies that there exists some $J\in\zz_+$ independent of $b$ such that, for any $j\in[J+1,\,\fz)\cap\zz_+$,
\begin{align}\label{eee.5}
  \lf\{\ln{\frac{2^{j+1}+1}{2^j-1}}+\frac{[(2^{j+1}+1)^{\rho}-(2^{j}-1)^{\rho}]^2}{2\rho(2^{j+1}+1)^{2\rho}}\r\}^{1/2}<1.
\end{align}
Therefore, for any $\lz\in(0,\,\fz)$, write
\begin{align*}
\int_{\rn}\vz\lf(x,\,\frac{\mu^\rho_\Omega(b)(x)}{\lz}\r)\,dx
=\int_{2^JB_{r}}\vz\lf(x,\,\frac{\mu^\rho_\Omega(b)(x)}{\lz}\r)\,dx
+\int_{\lf(2^JB_{r}\r)^\complement}\cdot\cdot\cdot
=:\mathrm{I_1+I_2}.
\end{align*}

Another step in the proof is to estimate $\mathrm{I_1}$ and $\mathrm{I_2}$, respectively.

For $\mathrm{I_1}$, by the uniformly upper type 1 property of $\vz$,
Theorem A with $\Omega\in L^q(S^{n-1})$ and $\vz^{q'}\in\mathbb{A}_d$,
and Lemma \ref{bcd} with $\vz\in\mathbb{A}_d$
(which is guaranteed by Lemma \ref{quan}(i) with \eqref{eee.4}), we know that,
for any $\lz\in(0,\,\fz)$,
\begin{align*}
{\mathrm{I_1}}
&\ls \int_{2^JB_{r}}\lf(1+\frac{\mu^\rho_\Omega(b)(x)}{\|b\|_{L^\fz}}\r)^d
\vz\lf(x,\,\frac{\|b\|_{L^\fz}}{\lz}\r)\,dx \\
&\ls \int_{2^JB_{r}}\lf(1+\frac{\lf[\mu^\rho_\Omega(b)(x)\r]^d}{\|b\|^d_{L^\fz}}\r)
\vz\lf(x,\,\frac{\|b\|_{L^\fz}}{\lz}\r)\,dx \\
&\ls \vz\lf(2^JB_{r},\,{\|b\|_{L^\fz}}\r)+\frac{1}{\|b\|^d_{L^\fz}}\int_\rn \lf[\mu^\rho_\Omega(b)(x)\r]^d \vz\lf(x,\,\frac{\|b\|_{L^\fz}}{\lz}\r)\,dx \\
&\ls \vz\lf(2^JB_{r},\,{\|b\|_{L^\fz}}\r)+\frac{1}{\|b\|^d_{L^\fz}}\int_{B_r} |b(x)|^d \vz\lf(x,\,\frac{\|b\|_{L^\fz}}{\lz}\r)\,dx \\
&\ls \vz\lf(B_{r},\,\frac{\|b\|_{L^\fz}}{\lz}\r),
\end{align*}
which is wished.

For ${\mathrm{I_2}}$, from the uniformly lower type $p$ properties of $\vz$ with $\frac{\mu^\rho_\Omega(b)(x)}{\|b\|_{L^\fz}}\ls1$ (see \eqref{eee.x} and \eqref{eee.5})
and H\"{o}lder's inequality, we deduce that, for any $\lz\in(0,\,\fz)$,
\begin{align*}
{\mathrm{I_2}}
&= \sum_{j=J+1}^{\fz}\int_{E_j}\vz\lf(x,\,\frac{\mu^\rho_\Omega(b)(x)}{\lz}\r)\,dx \\
&\ls \frac{1}{\|b\|^p_{L^\fz}}\sum_{j=J+1}^{\fz}\int_{E_j}\lf[\mu^\rho_\Omega(b)(x)\r]^p
\vz\lf(x,\,\frac{\|b\|_{L^\fz}}{\lz}\r)\,dx \\
&\ls \frac{1}{\|b\|^p_{L^\fz}}\sum_{j=J+1}^{\fz}
\lf(\int_{E_j}\lf[\vz\lf(x,\,\frac{\|b\|_{L^\fz}}{\lz}\r)\r]^{1/(1-p)}\,dx\r)^{1-p}
\lf(\int_{E_j}\mu^\rho_\Omega(b)(x)\,dx\r)^p.
\end{align*}
Notice that $\vz^{1/{(1-p)}}\in\aa_d\subset\aa_\fz\ $(see \eqref{eee.4}).
By Lemma \ref{fh}, we have $\vz\in \mathbb{RH}_{1/(1-p)}$. Thus,
from Lemma \ref{bcd} with $\vz^{1/{(1-p)}}\in\mathbb{A}_d$, and $\vz\in \mathbb{RH}_{1/(1-p)}$,
it follows that, for any $\lz\in(0,\,\fz)$,
\begin{align*}
\lf(\int_{E_j}\lf[\vz\lf(x,\,\frac{\|b\|_{L^\fz}}{\lz}\r)\r]^{1/(1-p)}\,dx\r)^{1-p}
&\leq \lf[\vz^{1/(1-p)} \lf(2^{j+1}B_{r},\,\frac{\|b\|_{L^\fz}}{\lz}\r)\r]^{1-p} \\
&\ls 2^{jnd(1-p)}\lf[\vz^{1/(1-p)}\lf(B_r,\,\frac{\|b\|_{L^\fz}}{\lz}\r)\r]^{1-p} \\
&\ls 2^{jnd(1-p)}r^{-np}\vz\lf(B_r,\,\frac{\|b\|_{L^\fz}}{\lz}\r).
\end{align*}
Since $d<{p\beta}/{[n(1-p)]}$, we may choose an $\tilde{\alpha}\in(0,\,\alpha)$
such that $d<{p\tilde{\beta}}/{[n(1-p)]}$, where $\tilde{\beta}:=\min\{\tilde{\alpha},\,1/2\}$. By the assumption $\Omega\in{\rm{Din}}^1_\alpha(S^{n-1})$,
$\Omega$ satisfies the $L^1$-Dini type condition of order $\tilde{\alpha}$.
Applying Lemma \ref{lin}(i), we obtain
$$\int_{E_j}\mu^\rho_\Omega(b)(x)\,dx
\ls \|b\|_{L^\fz}\lf(2^jr\r)^n \lf(\frac{r}{2^jr}\r)^{n+\tilde{\beta}}
\thicksim \|b\|_{L^\fz}r^n 2^{-j\tilde{\beta}}.
$$
Substituting the above two inequalities into ${\mathrm{I_2}}$, we know that, for any $\lz\in(0,\,\fz)$,
$${\mathrm{I_2}}
\ls \vz\lf(B_r,\,\frac{\|b\|_{L^\fz}}{\lz}\r)\lf(\sum_{j=J+1}^{\fz}2^{j(nd-ndp-p\tilde{\beta})}\r)
\ls \vz\lf(B_r,\,\frac{\|b\|_{L^\fz}}{\lz}\r),$$
where the last inequality is due to $d<{p\tilde{\beta}}/{[n(1-p)]}$.

Finally, combining the estimates of ${\rm{I_1}}$ and ${\rm{I_2}}$, we obtain the desired inequality.
This finishes the proof of Lemma \ref{lemma.2}.
\end{proof}

\begin{proof}[Proof of Theorem \ref{dingli.1}]
From Theorem A with $\omega\equiv1$,
it follows that $\mu^\rho_\Omega$ is bounded on $L^2$.
By this, Lemma \ref{lemma.2} and the fact that $\mu^\rho_\Omega$ is a positive sublinear operator,
applying Theorem \ref{yt} with $q=\fz$,
we know that $\mu^\rho_\Omega$ extends uniquely to a bounded operator from $H^\vz$ to $L^\vz$.
This finishes the proof of Theorem \ref{dingli.1}.
\end{proof}

\begin{proof}[Proof of Theorem \ref{dingli.2}]
The proof of Theorem \ref{dingli.2} is similar to that of Theorem \ref{dingli.1}. We only need
to modify the estimate of ${\mathrm{I_2}}$ in the proof of Lemma \ref{lemma.2}.
And fortunately, the estimate of ${\mathrm{I_2}}$ is nearly identity to that of
$J$ in the proof of \cite[Theorem 1.5]{ll07}, where \cite[Lemma 4.4(a)]{ll07} is used in that
proof, and
here Lemma \ref{lin}(ii) is used instead.
We leave the details to the interested readers.
\end{proof}

\begin{proof}[Proof of Corollary \ref{tuilun.1}]
By Lemma \ref{quan}(ii) with $\vz\in\aa_{p(1+\frac{\beta}{n})}$, we see that there exists some $d\in(1,\,\fz)$ such that $\vz^d\in\aa_{p(1+\frac{\beta}{n})}$.
For any $q\in(1,\,\fz)$, by $p>n/{(n+\beta)}$, we have $(p+p\beta/n-1/q)q'>p(1+\beta/n)$
and hence $\vz^d\in\aa_{(p+\frac{p\beta}{n}-\frac1q)q'}$.
Thus, we may choose $q:=d/{(d-1)}$ such that
$$\vz^{q'}=\vz^d\in\aa_{(p+\frac{\beta}{n}-\frac{1}{q})q'}$$
and hence Corollary \ref{tuilun.1} follows from Theorem \ref{dingli.2}.
\end{proof}

\begin{proof}[Proof of Theorem \ref{dingli.3}]
Observe that, if $\vz$ is of uniformly lower type 1 and of uniformly upper type 1,
then, in either $s\in(0,\,1]$ or $s\in[1,\,\infty)$, there exists a positive constant $C$ independent of $s$ such that, for any $x\in\rn$ and $t\in[0,\,\fz)$,
\begin{align}\label{xydy}
\vz(x,\,st)\le C s\vz(x,\,t).
\end{align}
On the other hand, we claim that, in either $s\in(0,\,1]$ or $s\in[1,\,\infty)$, there exists a positive constant $C$ independent of $s$ such that, for any $x\in\rn$ and $t\in[0,\,\fz)$,
\begin{align}\label{dydy}
\vz(x,\,st)\geq Cs\vz(x,\,t).
\end{align}
In fact, by \eqref{xydy}, we have
$$s\vz(x,\,t)=s\vz(x,\,st/s)\ls\vz(x,\,s\,t),$$
which is wished.
Combining \eqref{xydy} and \eqref{dydy}, we obtain $\vz(x,\,st)\sim s\vz(x,\,t)$,
which implies that
$$\int_\rn \vz\lf(x,\, f^\ast(x)\r)\, dx\thicksim\int_\rn f^\ast(x)\vz\lf(x,\,1\r)\, dx$$
and hence, $H^\vz=H^1_{\vz(\cdot\,,\,1)}$. Similarly, $L^\vz=L^1_{\vz(\cdot\,,\,1)}$.
Then, by repeating the proof of \cite[Theorem 1.8]{ll07}, we know that
$\lf\|\mu_{\Omega}^{\rho}(f)\r\|_{L^\vz} \ls\|f\|_{H^\vz}.$
This finishes the proof of Theorem \ref{dingli.4}.
\end{proof}

\begin{lemma}\label{m11}
Let $\rho\in(0,\,\infty)$, $\alpha\in(0,\,1]$, $\beta:=\min\{\alpha,\,1/2\}$ and
$\vz$ be a growth function as in Definition \ref{d2.3} with $p:={n}/{(n+\beta)}$ therein.
Suppose that $\Omega\in{\rm{Lip}}_\alpha(S^{n-1})$.
If $\vz\in\aa_1$, then there exists a positive constant $C$ such that,
for any $\lz\in(0,\,\fz)$ and multiple of a $(\vz,\,\fz,\,s)$-atom $b$ associated with some ball $B\subset\rn$,

$$\sup_{\alpha\in(0,\,\fz)}\vz\lf(\lf\{\mu^\rho_\Omega(b)>\alpha\r\},\,\frac{\alpha}{\lz}\r)
\le C\vz\lf(B,\,\frac{\|b\|_{L^\fz}}{\lz}\r).$$
\end{lemma}

\begin{proof}
We show this lemma by borrowing some ideas from the proof of \cite[Theorem 5.2]{lyj16}.
Without loss of generality, we may assume $b$ is a multiple of a
$(\vz,\,\fz,\,s)$-atom associated with a ball $B_r$ for some $r\in(0,\,\fz)$.
For the general case,
we refer the readers to the method of proof in \cite[Theorem 1.4]{ll07}.
Proceeding as in the proof of \cite[Theorem 1.1]{w16}, we know that, for any $x\in\lf(B_{2r}\r)^{\complement}$,
$$\mu^\rho_\Omega(b)(x)\ls \|b\|_{L^\fz}\lf(\frac{r^{n+1/2}}
{|x|^{n+1/2}}+\frac{r^{n+1}}{|x|^{n+1}}+\frac{r^{n+\alpha}}{|x|^{n+\alpha}}\r),$$
which, together with $\beta:=\min\{\alpha,\,1/2\}$, implies that, for any $x\in\lf(B_{2r}\r)^{\complement}$,
\begin{eqnarray}\label{brdtgj}
\mu^\rho_\Omega(b)(x)\ls \|b\|_{L^\fz}\frac{r^{n+\beta}}{|x|^{n+\beta}}\,.
\end{eqnarray}
Therefore, for any $\lz\in(0,\,\fz)$, write
\begin{align*}
&\sup_{\alpha\in(0,\,\fz)}\vz\lf(\lf\{\mu^\rho_\Omega(b)>\alpha\r\},\,\frac{\alpha}{\lz}\r) \\
&\hs\le \sup_{\alpha\in(0,\,\fz)}\vz\lf(\lf\{x\in B_{2r}: \mu^\rho_\Omega(b)(x)>\alpha\r\},\,\frac{\alpha}{\lz}\r) \\
&\hs\hs+\sup_{\alpha\in(0,\,\fz)}\vz\lf(\lf\{x\in (B_{2r})^\complement: \mu^\rho_\Omega(b)(x)>\alpha\r\},\,\frac{\alpha}{\lz}\r)
=:{\rm{I_1}}+{\rm{I_2}}.
\end{align*}

For $\mathrm{I_1}$, from Lemma \ref{quan}(ii) with $\vz\in\aa_2$ (since $\vz\in\aa_1$),
it follows that $\vz^{q'}\in\mathbb{A}_2$ for some ${q'}\in(1,\,\infty)$.
By the uniformly upper type 1 property of $\vz$,
Theorem A with $\Omega\in L^q(S^{n-1})$ (since $\Omega\in{\rm{Lip}}_\alpha(S^{n-1})$) and $\vz^{q'}\in\mathbb{A}_2$,
and Lemma \ref{bcd} with $\vz\in\mathbb{A}_2$, we know that, for any $\lz\in(0,\,\fz)$,
\begin{align*}
{\mathrm{I_1}}
&=\sup_{\alpha\in(0,\,\fz)}\int_{\lf\{x\in B_{2r}: \mu^\rho_\Omega(b)(x)>\alpha\r\}}\vz\lf(x,\,\frac{\alpha}{\lz}\r)\,dx\\
&\le \int_{B_{2r}}\vz\lf(x,\,\frac{\mu^\rho_\Omega(b)(x)}{\lz}\r)\,dx \\
&\ls \int_{B_{2r}}\lf(1+\frac{\mu^\rho_\Omega(b)(x)}{\|b\|_{L^\fz}}\r)^2\vz\lf(x,\,\frac{\|b\|_{L^\fz}}{\lz}\r)\,dx \\
&\ls \int_{B_{2r}}\lf(1+\frac{\lf[\mu^\rho_\Omega(b)(x)\r]^2}{\|b\|^2_{L^\fz}}\r)\vz\lf(x,\,\frac{\|b\|_{L^\fz}}{\lz}\r)\,dx \\
&\ls \vz\lf(B_{2r},\,\frac{\|b\|_{L^\fz}}{\lz}\r)+\frac{1}{\|b\|^2_{L^\fz}}\int_\rn \lf[\mu^\rho_\Omega(b)(x)\r]^2 \vz\lf(x,\,\frac{\|b\|_{L^\fz}}{\lz}\r)\,dx \\
&\ls \vz\lf(B_{2r},\,\frac{\|b\|_{L^\fz}}{\lz}\r)+\frac{1}{\|b\|^2_{L^\fz}}\int_{B_r} |b(x)|^2 \vz\lf(x,\,\frac{\|b\|_{L^\fz}}{\lz}\r)\,dx \\
&\ls \vz\lf(B_{r},\,\frac{\|b\|_{L^\fz}}{\lz}\r),
\end{align*}
which is wished.

For ${\rm{I_2}}$, from \eqref{brdtgj}, Lemma \ref{bcd} with $\vz\in\aa_1$,
and the uniformly lower type ${\frac{n}{n+\beta}}$ property of $\vz$, we deduce that, for any $\lz\in(0,\,\fz)$,
\begin{align*}
{\rm{I_2}}
&\ls\sup_{\alpha\in(0,\,\fz)}\vz\lf(\lf\{x\in (B_{2r})^\complement: \
\|b\|_{L^\fz}\frac{r^{n+\beta}}{|x|^{n+\beta}}>\alpha\r\},\,\frac{\alpha}{\lz}\r) \\
&\sim\sup_{\alpha\in(0,\,\fz)}\vz\lf(\lf\{x\in (B_{2r})^\complement: \
|x|^{n+\beta}<\frac{\|b\|_{L^\fz}}{\alpha}r^{n+\beta}\r\},\,\frac{\alpha}{\lz}\r) \\
&\sim\sup_{\alpha\in(0,\,\fz)}\vz\lf(\lf\{x\in \rn: \
2r\le|x|<\lf(\frac{\|b\|_{L^\fz}}{\alpha}\r)^{\frac{1}{n+\beta}} r \r\},\,\frac{\alpha}{\lz}\r) \\
&\ls\sup_{\alpha\in(0,\,\|b\|_{L^\fz})}\vz\lf(\lf\{x\in \rn: \
|x|<\lf(\frac{\|b\|_{L^\fz}}{\alpha}\r)^{\frac{1}{n+\beta}} r \r\},\,\frac{\alpha}{\lz}\r) \\
&\thicksim\sup_{\alpha\in(0,\,\|b\|_{L^\fz})}
\vz\lf(\lf[\frac{\|b\|_{L^\fz}}{\alpha}\r]^{\frac{1}{n+\beta}} B_r ,\,\frac{\alpha}{\lz}\r) \\
&\ls\sup_{\alpha\in(0,\,\|b\|_{L^\fz})} \lf(\frac{\|b\|_{L^\fz}}{\alpha}\r)^{\frac{n}{n+\beta}}
\vz\lf(B_r ,\,\frac{\alpha}{\lz}\r) \\
&\ls\sup_{\alpha\in(0,\,\|b\|_{L^\fz})} \lf(\frac{\|b\|_{L^\fz}}{\alpha}\r)^{\frac{n}{n+\beta}}
\lf(\frac{\alpha}{\|b\|_{L^\fz}}\r)^{\frac{n}{n+\beta}} \vz\lf(B_r,\,\frac{\|b\|_{L^\fz}}{\lz}\r) \\
&\thicksim \vz\lf(B_r,\,\frac{\|b\|_{L^\fz}}{\lz}\r).
\end{align*}

Combining the estimates of ${\rm{I_1}}$ and ${\rm{I_2}}$, we obtain the desired inequality.
This finishes the proof of Lemma \ref{m11}.
\end{proof}

\begin{proof}[Proof of Theorem \ref{dingli.4}]
From Theorem A with $\omega\equiv1$,
it follows that $\mu^\rho_\Omega$ is bounded on $L^2$.
By this, Lemma \ref{lemma.2} and the fact that $\mu^\rho_\Omega$ is a positive sublinear operator,
applying Theorem \ref{yt2} with $q=\fz$,
we know that $\mu^\rho_\Omega$ extends uniquely to a bounded operator from ${H^\vz}$ to $WL^\vz$.
This finishes the proof of Theorem \ref{dingli.4}.
\end{proof}

\medskip

\noindent{\bf Competing Interests}

\medskip
\noindent The authors declare that there is no conflict of interests regarding the publication of this paper.

\medskip

\noindent {\bf Authors' Contributions}

\medskip

\noindent Li Bo conceived of the study. Liu Xiong, Li Baode, Qiu Xiaoli and Li Bo carried out the main results, participated in the sequence alignment and drafted the manuscript. Moreover, all authors read and approved the final manuscript.

\medskip

\noindent {\bf Acknowledgements}

\medskip

\noindent
This work is partially supported by the National Natural Science Foundation of China (Grant Nos. 11461065 \& 11661075)
and A Cultivate Project for Young Doctor from Xinjiang Uyghur Autonomous Region (No. qn2015bs003).
The authors would like to thank the anonymous referees for their constructive comments.

\bigskip

\noindent  Liu Xiong, Li Baode, Qiu Xiaoli and Li Bo (Corresponding author)

\medskip

\noindent
College of Mathematics and System Sciences \\
Xinjiang University\\
Urumqi 830046\\
P. R. China

\smallskip

\noindent{E-mail }:\\
\texttt{1394758246@qq.com} (Liu Xiong)  \\
\texttt{1246530557@qq.com} (Li Baode)  \\
\texttt{2237424863@qq.com} (Qiu Xiaoli) \\
\texttt{bli.math@outlook.com} (Li Bo)


%
%
%
%

\bigskip


\begin{thebibliography}{30}
\vspace{-0.3cm}
\bibitem{ak14}
Akbulut, Ali; Kuzu, Okan,
Marcinkiewicz integrals associated with Schr\"{o}dinger operator on generalized Morrey spaces,
J. Math. Inequal. 8 (2014), no. 4, 791-801.

\vspace{-0.3cm}
\bibitem{bckyy13}
Bui, The Anh; Cao, Jun; Ky, Luong Dang; Yang, Dachun; Yang, Sibei,
Musielak-Orlicz-Hardy spaces associated with operators satisfying reinforced off-diagonal estimates,
Anal. Geom. Metr. Spaces 1 (2013), 69-129.

\vspace{-0.3cm}
\bibitem{ccyy16}
Cao, Jun; Chang, Der-Chen; Yang, Dachun; Yang, Sibei,
Riesz transform characterizations of Musielak-Orlicz Hardy spaces,
Trans. Amer. Math. Soc. 368 (2016),  no. 10, 6979-7018.

\vspace{-0.3cm}
\bibitem{d05} Diening, Lars, Maximal function on Musielak-Orlicz spaces and generalized Lebesgue
spaces, Bull. Sci. Math. 129 (2005), no. 8, 657-700.

\vspace{-0.3cm}
\bibitem{d09} Diening, Lars; H\"{a}st\"{o}, Peter A.; Roudenko, Svetlana, Function spaces of variable smoothness
and integrability, J. Funct. Anal. 256 (2009), no. 6, 1731-1768.


\vspace{-0.3cm}
\bibitem{fhly17}
Fan, Xingya; He, Jianxun; Li, Baode; Yang, Dachun,
Real-variable characterizations of anisotropic product Musielak-Orlicz Hardy spaces,
Sci. China Math. 60 (2017), doi: 10.1007/s11425-016-9024-2, to appear.

\vspace{-0.3cm}
\bibitem{gahk16}
Guliyev, Vagif S.; Akbulut, Ali; Hamzayev, Vugar H.; Kuzu, Okan,
Commutators of Marcinkiewicz integrals associated with Schr\"{o}dinger operator on generalized weighted Morrey spaces,
J. Math. Inequal. 10 (2016), no. 4, 947-970.

\vspace{-0.3cm}
\bibitem{g09c}
Grafakos, Loukas, Classical Fourier Analysis, Second edition, Graduate Texts in Mathematics Vol. 249
(Springer, New York, 2009).

\vspace{-0.3cm}
\bibitem{g09m}
Grafakos, Loukas, Modern Fourier Analysis, Second edition, Graduate Texts in Mathematics Vol. 250
(Springer, New York, 2009).

\vspace{-0.3cm}
\bibitem{h60}
H\"{o}rmander, Lars, Estimates for translation invariant operators in $L^p$ spaces, Acta Math. 104 (1960), 93-140.

\vspace{-0.3cm}
\bibitem{hyy13}
Hou, Shaoxiong; Yang, Dachun; Yang, Sibei, Lusin area function and molecular characterizations of Musielak-Orlicz Hardy spaces and their applications, Commun. Contemp. Math. 15 (2013), no. 6, 1350029, 37 pp.

\vspace{-0.3cm}
\bibitem{j80}
Janson, Svante, Generalizations of Lipschitz spaces and an application to Hardy spaces and bounded mean oscillation, Duke Math. J. 47 (1980), no. 4, 959-982.

\vspace{-0.3cm}
\bibitem{jy10}
Jiang, Renjin; Yang, Dachun, New Orlicz-Hardy spaces associated with divergence form elliptic operators, J. Funct. Anal. 258 (2010), no. 4, 1167-1224.

\vspace{-0.3cm}
\bibitem{jn87}
Johnson, Raymond L.; Neugebauer, Christoph Johannes, Homeomorphisms preserving $A_p$, Rev. Mat. Iberoamericana 3 (1987), no. 2, 249-273.

\vspace{-0.3cm}
\bibitem{kw79}
Kurtz, Douglas S.; Wheeden, Richard L., Results on weighted norm inequalities for multipliers, Trans. Amer. Math. Soc. 255 (1979), 343-362.

\vspace{-0.3cm}
\bibitem{k14}
Ky, Luong Dang, New Hardy spaces of Musielak-Orlicz type and boundedness of sublinear operators, Integral Equations Operator Theory 78 (2014), no. 1, 115-150.



\vspace{-0.3cm}
\bibitem{lfy15}
Li, Baode; Fan, Xingya; Yang, Dachun, Littlewood-Paley characterizations of anisotropic Hardy spaces of Musielak-Orlicz type, Taiwanese J. Math. 19 (2015), no. 1, 279-314.

\vspace{-0.3cm}
\bibitem{lffy16}
Li, Baode; Fan, Xingya; Fu, Zunwei; Yang, Dachun, Molecular characterization of
anisotropic Musielak-Orlicz Hardy spaces and their applications,
Acta Math. Sin. (Engl. Ser.) 32 (2016), no. 11, 1391-1414.

\vspace{-0.3cm}
\bibitem{lsl16}
Li, Jinxia; Sun, Ruirui; Li, Baode, Anisotropic interpolation theorems of Musielak-Orlicz type, J. Inequal. Appl. 2016, 2016: 243.

\vspace{-0.3cm}
\bibitem{lhy12}
Liang, Yiyu; Huang, Jizheng; Yang, Dachun, New real-variable characterizations of Musielak-Orlicz Hardy spaces, J. Math. Anal. Appl. 395 (2012), no. 1, 413-428.

\vspace{-0.3cm}
\bibitem{ly13}
Liang, Yiyu; Yang, Dachun, Musielak-Orlicz Campanato spaces and applications,
J. Math. Anal. Appl. 406 (2013), no. 1, 307-322.

\vspace{-0.3cm}
\bibitem{lyj16}
Liang, Yiyu; Yang, Dachun; Jiang, Renjin, Weak Musielak-Orlicz Hardy spaces and applications,
Math. Nachr. 289 (2016), no. 5-6, 634-677.


\vspace{-0.3cm}
\bibitem{ll07}
Lin, Chin-Cheng; Lin, Ying-Chieh, $H^p_\omega-L^p_\omega$ boundedness of Marcinkiewicz integral, Integral Equations Operator Theory 58 (2007), no. 1, 87-98.

\vspace{-0.3cm}
\bibitem{lll17}
Li, Bo; Liao, Minfeng; Li, Baode,
Boundedness of Marcinkiewicz integrals with rough kernels on Musielak-Orlicz Hardy spaces,
J. Inequal. Appl. 2017, 2017: 228.


\vspace{-0.3cm}
\bibitem{m38}
Marcinkiewicz, J., Sur quelques int\'egrales du type de Dini,
Annales de la Soci\'et\'e Polonaise de Math\'ematiques, 17 (1938), 42-50.



\vspace{-0.3cm}
\bibitem{sj09}
Shi, Xinfeng; Jiang, Yinsheng, Weighted boundedness of parametric Marcinkiewicz integral and higher order commutator,
Anal. Theory Appl. 25 (2009), no. 1, 25-39.

\vspace{-0.3cm}
\bibitem{s84}
Stefan, Rolewicz, Metric Linear Spaces, Second edition, (PWN-Polish Scientific Publishers, Warsaw, 1984).

\vspace{-0.3cm}
\bibitem{s58}
Stein, Elias M., On the functions of Littlewood-Paley, Lusin, and Marcinkiewicz, Trans. Amer. Math. Soc. 88 (1958), 430-466.

\vspace{-0.3cm}
\bibitem{st89}
Str\"omberg, Jan-Olov; Torchinsky, Alberto, Weighted Hardy spaces, Lecture Notes in Mathematics Vol. 1381 (Springer-Verlag, Berlin, 1989).

\vspace{-0.3cm}
\bibitem{sw85}
Str\"omberg, Jan-Olov; Wheeden, Richard L., Fractional integrals on weighted $H^p$ and $L^p$ spaces,
Trans. Amer. Math. Soc. 287 (1985), no. 1, 293-321.


\vspace{-0.3cm}
\bibitem{w16}
Wang, Hua, Parametric Marcinkiewicz integrals on the weighted Hardy and weak Hardy spaces,
J. Math. Inequal. 10 (2016), no. 2, 373-391.

\vspace{-0.3cm}
\bibitem{ylk17}
Yang, Dachun; Liang, Yiyu; Ky, Luong Dang,
Real-Variable Theory of Musielak-Orlicz Hardy spaces, Lecture Notes in Mathematics Vol. 2182
(Springer, Cham, 2017).


\end{thebibliography}
\end{document}